\documentclass[10pt, leqno]{article}

\usepackage{amsmath,amssymb,amsthm, epsfig}

\usepackage{amsmath}

\usepackage{lineno}

\usepackage{comment}

\usepackage[pdftex]{color}

\usepackage{comment}

\usepackage{stackrel}

\usepackage{mathtools}
\usepackage{color}

\usepackage{ulem}

\usepackage{mathrsfs}
\usepackage[mathscr]{euscript}

\usepackage{wasysym}

\usepackage{stackrel}

\usepackage{amsmath,amssymb,amsthm, epsfig}
\usepackage[colorlinks=true,urlcolor=blue,
citecolor=red,linkcolor=blue,linktocpage,pdfpagelabels,
bookmarksnumbered,bookmarksopen]{hyperref}
\usepackage{ulem}
\usepackage{dsfont}
\usepackage{mathrsfs}
\usepackage{mathtools}
\usepackage{stackrel}
\usepackage{esint,enumerate}

\usepackage[left=2.5cm,right=2.5cm,top=2.5cm,bottom=2.5cm]{geometry}

\newcommand{\defeq}{\mathrel{\mathop:}=}

\newcommand{\intav}[1]{\mathchoice {\mathop{\vrule width 6pt height 3 pt depth  -2.5pt
\kern -8pt \intop}\nolimits_{\kern -6pt#1}} {\mathop{\vrule width
5pt height 3  pt depth -2.6pt \kern -6pt \intop}\nolimits_{#1}}
{\mathop{\vrule width 5pt height 3 pt depth -2.6pt \kern -6pt
\intop}\nolimits_{#1}} {\mathop{\vrule width 5pt height 3 pt depth
-2.6pt \kern -6pt \intop}\nolimits_{#1}}}

\newtheorem{theorem}{Theorem}[section]

\newtheorem{lemma}[theorem]{Lemma}

\newtheorem{proposition}[theorem]{Proposition}

\newtheorem{corollary}[theorem]{Corollary}

\theoremstyle{definition}

\newtheorem{definition}[theorem]{Definition}

\theoremstyle{remark}

\newtheorem{remark}[theorem]{Remark}

\numberwithin{equation}{section}

\usepackage{pgfplots}
\pgfplotsset{compat=newest}
\title{Schauder estimates for flat solutions to a class of fully nonlinear elliptic PDEs with Dini continuous data: a geometric tangential approach}
\author{\it by \smallskip \\ Junior da Silva Bessa \footnote{\noindent Universidade Estadual de Campinas - UNICAMP. Instituto de Matem\'{a}tica, Estat\'{i}stica e Computa\c{c}\~{a}o Cient\'{i}fica - IMECC. Departamento  de Matemática. Bar\~{a}o Geraldo, Campinas - SP, Brazil. \noindent \texttt{E-mail address: \url{jbessa@unicamp.br}}}, \quad Jo\~{a}o Vitor  da Silva
\footnote{\noindent Universidade Estadual de Campinas - UNICAMP Instituto de Matem\'{a}tica, Estat\'{i}stica e Computa\c{c}\~{a}o Cient\'{i}fica - IMECC. Departamento  de Matemática. Bar\~{a}o Geraldo, Campinas - SP, Brazil. \noindent \texttt{E-mail address: \url{jdasilva@unicamp.br}}},\\ \quad $\&$ \\\quad Laura Ospina\footnote{\noindent Universidade Estadual de Campinas - UNICAMP. Instituto de Matem\'{a}tica, Estat\'{i}stica e Computa\c{c}\~{a}o Cient\'{i}fica - IMECC. Departamento  de Matemática. Bar\~{a}o Geraldo, Campinas - SP, Brazil. \noindent \texttt{E-mail address: \url{l202049@dac.unicamp.br}}}
}
\date{}

\begin{document}

\maketitle

\begin{abstract}
\noindent  In this manuscript, we establish local Schauder estimates for flat viscosity solutions, that is, solutions with sufficiently small norms, to a class of fully nonlinear elliptic partial differential equations of the form
\[
F(D^{2} u, x) + \langle \mathfrak{B}(x), D u \rangle = f(x) \quad \text{in} \quad \mathrm{B}_1 \subset \mathbb{R}^{n},
\]
where the operator \(F\) is differentiable, though not necessarily convex or concave. In addition, we impose suitable Dini-type continuity assumptions on the data. Our methodology is based on geometric tangential techniques, combined with compactness and perturbative arguments. This approach is strongly motivated by recent advances in the theory of nonlinear elliptic equations and free boundary problems. As a byproduct of our analysis, we also obtain an Evans–Krylov type estimate. Our results can be viewed as an extension of the work by dos Prazeres and Teixeira~\cite[Theorem 2.2]{dosPrazTei2016}, now within the framework of linear drift terms and Dini continuity assumptions. Finally, we apply our results to characterize the nodal sets of flat viscosity solutions of non-convex, fully nonlinear, uniformly elliptic PDEs.

\vspace{0.2cm}
	

\end{abstract}

\section{Introduction}

It is a well-known fact in the current literature that viscosity solutions to fully nonlinear elliptic equations 
\begin{equation}\label{Eq1.1}
F(D^2u) = 0 \quad \text{in} \quad \mathrm{B}_1,
\end{equation}
for a second-order, uniformly elliptic operator $F: \text{Sym}(n) \to \mathbb{R}$, are locally $C_{\text{loc}}^{1,\alpha_{\text{Hom}}}$ for some $\alpha_{\text{Hom}} \in (0,1)$ depending only on dimension and ellipticity constants $0< \lambda \le \Lambda$ (see \cite{Caffarelli1989} and \cite[Ch.~5]{CafCabre1995}). Regarding classical, i.e., $C^2$, solutions, the fundamental and independent contributions of Evans \cite{Evans1982} and Krylov \cite{Krylov1983} remain particularly significant. More precisely, for concave or convex operators $F$, they proved $C_{\text{loc}}^{2,\alpha}$ regularity for solutions to \eqref{Eq1.1} (see also \cite[Ch.~6]{CafCabre1995}).

\begin{theorem}[{\bf Evans--Krylov's Theorem}]\label{thm:evans-krylov}
Let \( F: \text{Sym}(n) \to \mathbb{R} \) be a convex (or concave) uniformly elliptic operator satisfying \( F(\mathbf{O}_n) = 0 \) (where $\mathbf{O}_{n}$ denotes the $n \times n$ zero symmetric matrix). Let \( u \in C^0(\mathrm{B}_1) \) be a viscosity solution of
\[
F(D^2 u) = 0 \quad \text{in } \quad  \mathrm{B}_1.
\]
Then,
\[
\|u\|_{C^{2,\alpha_{\mathrm{EK}}}(\mathrm{B}_{1/2})} \leq \mathrm{C} \|u\|_{L^{\infty}(\mathrm{B}_1)},
\]
for some \( \alpha_{\mathrm{EK}}  \in (0, 1) \) and constant \( \mathrm{C} > 0 \), depending only on the dimension \( n \), and the ellipticity constants \( \lambda \) and \( \Lambda \). Particular, if \( F \) is smooth, then \( u \in C^\infty(\mathrm{B}_1) \).
\end{theorem}

A few years later, Caffarelli established Schauder estimates for the inhomogeneous problem via a perturbation and compactness technique under suitable H\"{o}lder assumptions on data, see \cite[Theorem 8.1]{CafCabre1995}.  Precisely, if \( u \in C^0(\mathrm{B}_1) \) is a viscosity solution to
\[
F(D^2 u, x) = f(x) \quad \text{in } \mathrm{B}_1 \subset \mathbb{R}^n, \quad \theta_F, f \in C^{0, \alpha}(\mathrm{B}_1) \quad  \text{for some}\,\,\,  \alpha \in (0,1).
\]
Here, $\theta_{F}$ denotes the oscillation of the coefficients of the operator $F$ (see \eqref{defosccoef} for the precise definition). Then, \( u \in C^{2,\min\{\alpha_{\mathrm{EK}}^{-}, \alpha\}}(\mathrm{B}_{1/2}) \), and
    \[
    \|u\|_{C^{2,\min\{\alpha_{\mathrm{EK}}^{-},\alpha\}}(\mathrm{B}_{1/2})} \leq \mathrm{C}\left( \|u\|_{L^{\infty}(\mathrm{B}_1)} + \|f\|_{C^{0, \alpha}(\mathrm{B}_1)} \right),
    \]
    for some constant \( \mathrm{C} > 0 \) depending only on \( n, \lambda \), \( \Lambda \), and \( \alpha \).

\bigskip

For nearly twenty years, the question of whether \textit{arbitrary} fully nonlinear elliptic equations $F(D^2 \mathfrak{h}) = 0$ in $\mathrm{B}_1$ (resp. $F(D^2 \mathfrak{h}) = f(x)$) admit a general $C^2$ \textit{a priori} regularity theory remained unresolved. This problem was finally settled by Nadirashvili and Vl\u{a}du\c{t}'s counterexamples to $C^{1,1}$ regularity in \cite{NV2007} and \cite{NV2008}, which concluded this line of investigation. Their works, however, stimulated new research directions. In light of the inherent obstacles to developing a universal existence theory for classical solutions to fully nonlinear equations, a major focus of current research involves identifying supplementary structural or qualitative conditions on $F(\mathrm{X}, \cdot), F(\cdot, x), f$, and even on $u$ that enable $C^2$ estimates.

\medskip

In this context, in 2016, dos Prazeres-Teixeira in 
\cite{dosPrazTei2016} addressed Schauder estimates for flat solutions of non-convex, fully nonlinear, uniformly elliptic PDEs under H\"{o}lder continuity assumption on data (cf. \cite{daSdosP2019}).

\begin{theorem}[{\bf  $C^{2,\alpha}$ regularity - \cite[Theorem 2.2]{dosPrazTei2016}}] 
Let \( u \in C^0(\mathrm{B}_1) \) be a viscosity solution to
\begin{equation}\label{Eq-dosPT}
F(D^2u,x) = f(x) \quad \text{in }  \quad \mathrm{B}_1,    
\end{equation}
where \( F \) and \( f \) satisfy $(\mathrm{A1})–(\mathrm{A4})$ with \( \tau(t) = \mathrm{c}t^\alpha \) for some \( 0 < \alpha < 1 \). There exists \( \delta_0 > 0 \), depending only upon \( n, \lambda, \Lambda, \omega, \alpha \), and \( \tau(1) \), such that if
\begin{equation}\tag{{\bf Flat condition}}
\sup_{\mathrm{B}_1} |u| \leq \delta_0,
\end{equation}
then \( u \in C^{2,\alpha}(\mathrm{B}_{1/2}) \) and
\[
\|u\|_{C^{2,\alpha}(\mathrm{B}_{1/2})} \leq \mathrm{C}_0 \cdot \delta_0,
\]
where \( \mathrm{C}_0>0\) depends only upon \( n, \lambda, \Lambda, \omega \), and \( (1 - \alpha) \).
    
\end{theorem}

\medskip
Therefore, motivated by the previous presentation, we establish local Schauder estimates for flat viscosity solutions to the equation
\begin{equation}\label{MainEq}
F(D^{2}u,x) +  \langle \mathfrak{B}(x), Du\rangle = f(x) \quad \text{in} \quad \mathrm{B}_1 \subset \mathbb{R}^n,
\end{equation}
under suitable Dini continuity assumptions on the data, which will be specified shortly. From now on, $F:\text{Sym}(n)\times \mathrm{B}_1 \to \mathbb{R}$ denotes a fully nonlinear second-order operator, potentially including nonlinearities that lack a convexity/concavity structure. Furthermore, a \textbf{flat viscosity solution} means a solution to \eqref{MainEq} with a sufficiently small norm depending on universal and structural parameters.

\bigskip

Now, before presenting the main assumption in our manuscript, we recall that any modulus of continuity \( \varphi: [0, \infty) \to [0, \infty) \) is non-decreasing, sub-additive, continuous, and satisfies \( \varphi(0) = 0 \).

Furthermore, a modulus of continuity \( \varphi \) is called \textbf{Dini continuous} if it satisfies
\begin{equation}\label{DiniCondition}\tag{\textbf{DC}}
\mathfrak{I}_{\varphi} \defeq \int_0^1 \frac{\varphi(r)}{r} \, dr < +\infty.
\end{equation}

Additionally,  we consider the function 
\begin{equation}\label{defosccoef}
\theta(x, x_0) = \theta_F(x, x_0) \defeq \sup_{\mathrm{X} \in \text{Sym}(n)} \frac{|F(\mathrm{X},x) - F(\mathrm{X},x_0)|}{1+\|\mathrm{X}\|},    
\end{equation}
which quantifies the oscillation of \( F \) in the variables \(x\) near the $x_0$ (i.e., the oscillation of the coefficients). Moreover, by simplicity of notation, we will write $\theta_{F}(x)$ to denote the oscillation around the origin.

\medskip

Throughout this manuscript, we assume the following structural assumptions (cf. \cite{Bao2002}, \cite{daSdosP2019}, \cite{dosPrazTei2016}, and \cite{ZC2002}):
\begin{itemize}
\item[{\bf ($\mathrm{A}1$)}] ({\bf Uniform Ellipticity}) The operator $F:\text{Sym}(n)\times \Omega \to \mathbb{R}$ is fully nonlinear and uniformly elliptic, with ellipticity constants $0 < \lambda \leq \Lambda$. Specifically, we require
\begin{equation}\label{Unif_Elip}
\mathscr{P}^{-}_{\lambda, \Lambda}(\mathrm{N}) \leq F(\mathrm{M} + \mathrm{N},x) - F(\mathrm{M},x) \leq \mathscr{P}^{+}_{\lambda,\Lambda}(\mathrm{N}),
\end{equation}
for every $x \in \Omega$ and for all $\mathrm{M}, \mathrm{N} \in \text{Sym}(n)$ with $\mathrm{N} \geq 0$ (in the matrix sense), where
$$
\mathscr{P}_{\lambda, \Lambda}^+(\mathrm{M}) \coloneqq \sup_{\mathbf{A} \in \mathfrak{A}_{\lambda, \Lambda}} \operatorname{tr}(\mathbf{A} \mathrm{M}) \quad \text{and} \quad 
\mathscr{P}_{\lambda, \Lambda}^-(\mathrm{M}) \coloneqq \inf_{\mathbf{A} \in \mathfrak{A}_{\lambda, \Lambda}} \operatorname{tr}(\mathbf{A} \mathrm{M})
$$
denote the \textit{Pucci extremal operators}, and
$$
\mathfrak{A}_{\lambda, \Lambda} \coloneqq \{\mathbf{A} \in \text{Sym}(n): \lambda \mathrm{Id}_n \leq \mathbf{A} \leq \Lambda \mathrm{Id}_n\}.
$$
Additionally, we assume that $F(\mathbf{O}_n, x) = 0$ for every $x \in \mathrm{B}_1$.

\item[{\bf ($\mathrm{A}2$)}] ({\bf $C^{1}$-regularity of the Nonlinearity}) We assume that $F \in C^1(\text{Sym}(n))$, and there exists a modulus of continuity \(\omega: [0, \infty) \to [0, \infty)\) such that
$$
\|\mathrm{D}_{\mathrm{M}}F(\mathrm{X},x) - \mathrm{D}_{\mathrm{M}}F(\mathrm{Y},x)\| \leq \omega(\|\mathrm{X} - \mathrm{Y}\|),
$$
for all $x \in \Omega$ and for all \(\mathrm{X}, \mathrm{Y} \in \text{Sym}(n)\), 
where
$$
\mathrm{D}_{\mathrm{M}}F(\mathrm{X}_0,x) = \mathrm{D}F(\mathrm{X}_0,x)(\mathrm{M}) \coloneqq \lim_{t \to 0} \frac{F(t \mathrm{M} + \mathrm{X}_0,x) - F(\mathrm{X}_0,x)}{t}
$$
denotes the \textit{G\^{a}teaux derivative} of $F$ with respect to the variable \(\mathrm{M}\in \text{Sym}(n)\).

\item[{\bf ($\mathrm{A}3$)}] ({\bf Dini continuity of the Data in the $L^{p_{0}}$-sense}) We assume that $f \in C^0(\mathrm{B}_1)$, $\mathfrak{B} \in C^{0}(\mathrm{B}_1, \mathbb{R}^n)$, and $\theta_F \in C^0(\mathrm{B}_1)$. There exists a modulus of continuity $\tau: [0, \infty) \to [0, \infty)$, and positive constants $\mathrm{C}_f, \mathrm{C}_{\theta_F}, \mathrm{C}_{\mathfrak{B}}$, such that

$$
\displaystyle \left( \intav{\mathrm{B}_r(x_0)} |f(x)-f(x_0)|^{p_{0}}dx\right)^{\frac{1}{p_{0}}}\leq \mathrm{C}_f\tau(r), \quad \left( \intav{\mathrm{B}_r(x_0)} |\theta_{F}(x, x_0)|^{p_{0}}dx\right)^{\frac{1}{p_{0}}}\leq \mathrm{C}_{\theta_F}\tau(r),   
$$
and 
$$
\left( \intav{\mathrm{B}_r(x_0)} |\mathfrak{B}(x)-\mathfrak{B}(x_0)|^{p_{0}}dx\right)^{\frac{1}{p_{0}}}\leq \mathrm{C}_{\mathfrak{B}}\tau(r),
$$
where $\tau$ satisfies the Dini condition \eqref{DiniCondition} and \(p_{0}>n\).

\item[{\bf ($\mathrm{A}4$)}] ({\bf Compatibility condition}) We further suppose the \textbf{nullity conditions} for the modulus of continuity:
$$
\displaystyle (i) \quad \liminf_{s \to 0^{+}} \sup_{0 < r \leq \frac{1}{2}} \frac{\tau(rs)}{\tau(r)} = 0 \quad \text{and} \quad (ii) \quad \liminf_{s \to 0^{+}} \sup_{k \geq 0} \frac{s^{\alpha_{0}} \tau(s^k)}{\tau(s^{k+1})} = 0.
$$
for some \(\alpha_{0}\in(0,1]\).
\end{itemize}
\begin{remark}
For the sake of simplicity, we will denote by $\mathscr{F}$ the set of solutions to \eqref{MainEq}, whose data satisfy conditions $\rm{(A1)}$–$\rm{(A4)}$. Specifically,
\[
\mathscr{F}=\{u \in  C^0(\mathrm{B}_{1})\ |\ u \,\,\, \text{is a viscosity solution to}\,\,\,\eqref{MainEq}\,\,\,\text{such that}\,\,\,  \rm{(A1)}-\rm{(A4)} \,\,\,\text{are satisfied}\}.
\]
\end{remark}

\begin{remark}\label{Remark1.4}
Observe that ($\mathrm{A}4$) implies that $\displaystyle \lim_{s \to 0^+} \frac{s}{\tau(s)} = 0$. 
Moreover, the second assumption in $\mathrm{(A4)}$ is strictly weaker than the conditions considered in Kovats' works
\cite[Remark, \S1]{Kovats97} and \cite[Remark, \S3]{Kovats99}, namely,
\[
\lim_{s\to 0^{+}}
\sup_{0<r\leq \frac{1}{2}}
\frac{s^{\alpha_{0}}\bigl(r^{\alpha_{0}}+\tau(r)\bigr)}
{(rs)^{\alpha_{0}}+\tau(rs)}
=0.
\]
Indeed, under this hypothesis, for sufficiently small $s>0$, we have
\[
\frac{s^{\alpha_{0}}\tau(s^{k})}{\tau(s^{k+1})}
\leq
\frac{2s^{\alpha_{0}}\tau(s^{k})}{s^{(k+1)\alpha_{0}}+\tau(s^{k+1})}
\leq
\frac{2s^{\alpha_{0}}\bigl((s^{k})^{\alpha_{0}}+\tau(s^{k})\bigr)}
{(s^{k}s)^{\alpha_{0}}+\tau(s^{k}s)},
\]
which immediately yields the second nullity condition in $\mathrm{(A4)}$.

In addition, assumption $\mathrm{(A4)}$ is satisfied by Dini moduli of continuity that are not necessarily
$\gamma$--H\"older continuous for some suitable $\gamma\in(0,1)$.
For instance, it is straightforward to verify that
\[
\boxed{
\tau(r)=\frac{r^{\alpha}}{|\log r|^{\beta}},
\qquad
0<\alpha<\alpha_{0},\ \beta>0,
}
\]
fulfills all the assumptions in $\mathrm{(A4)}$, while
\[
\boxed{
\tau(r)=\frac{r^{\alpha}}{|\log r|^{\beta}}
\text{ is not }\gamma\text{--H\"older continuous}
\quad\Longleftrightarrow\quad
\gamma>\alpha.
}
\]
Along the same lines, one easily checks (see \cite[Remark~14]{Bao2002}) that
\[
\boxed{
\tau(t)= t^{\kappa}\ln^{\zeta}\!\Big(\tfrac{1}{t}\Big),
\qquad
0<\kappa<\alpha_{0},\ \zeta\neq0,
}
\]
satisfies both nullity conditions in $\mathrm{(A4)}$, and that
\[
\boxed{
\tau(t)= t^{\kappa}\ln^{\zeta}\!\Big(\tfrac{1}{t}\Big)
\text{ is not }\gamma\text{--H\"older continuous}
\iff
\begin{cases}
\gamma>\kappa, & \text{for any }\zeta\neq0,\\[4pt]
\gamma=\kappa \text{ and } \zeta>0.
\end{cases}
}
\]
Consequently, condition $\mathrm{(A4)}$ allows us to treat a genuine extension of Schauder-type estimates
for Dini moduli of continuity that need not be H\"older continuous.
\end{remark}

\medskip

\medskip

By way of comparison, when $\tau_{\alpha}(r) =  \mathrm{c}r^\alpha$ for some $\alpha \in (0, 1)$ and $\mathrm{c}>0$ (i.e., $\tau(r)=\tau_{\alpha}(r)$ in the assumption $(\text{A3})$), this setting corresponds to the Schauder estimates studied by Dos Prazeres and Teixeira in \cite[Theorem 2.2]{dosPrazTei2016}. Nevertheless, there exist weaker moduli of continuity that still satisfy the Dini condition, yet they do not exhibit H\"{o}lder continuity. For example, consider the function defined by $\varrho(r) = |\ln (r)|^{-\gamma}$ for $0 < r \leq 1/2$ and $\varrho(0) = 0$. This function $\varrho$ satisfies the Dini condition \eqref{DiniCondition} provided that $\gamma > 1$. However, $\varrho$ fails to be H\"{o}lder continuous at the origin for any $\alpha \in (0, 1)$, as
\[
\lim_{r \to 0^+} \frac{|\varrho(r)-\varrho(0)|}{r^{\alpha}} = +\infty.
\]
Consequently, the available Schauder regularity estimates for flat viscosity solutions to \eqref{Eq-dosPT} cannot be applied in the absence of a H\"{o}lder continuity assumption on the data (see also \cite{JVGN21} for the case of convex operators with Hamiltonian terms).

This observation naturally leads to the following question: Is it feasible to establish a Schauder theory for non-convex operators with data that are more general than H\"{o}lder continuous, yet satisfy an appropriate continuity condition?
\medskip

Within this framework, we are able to derive the following local Schauder estimates.

\begin{theorem}[{\bf Local $C^{2,\textrm{Dini}}$ regularity}]\label{MainThm01}
Let \( u \in \mathscr{F} \). There exists \( \delta_0 > 0 \), depending only upon \( n\), \(p_{0}\), \(\lambda\), \(\Lambda\), \(\omega\), and \( \tau(1) \), such that if
\[
\|u\|_{L^{\infty}(\mathrm{B}_1)} \leq \delta_0,
\]
then \( u \in C_{\text{loc}}^{2,\psi}(\mathrm{B}_{1}) \) and
\[
\|u\|_{C^{2,\psi}(\mathrm{B}_{1/2})} \leq \mathrm{C}_0 \cdot \delta_0,
\]
where \( \mathrm{C}_0>0\) depends only upon \( n\), \(p_{0}\), \(\lambda\), \(\Lambda\), \(\alpha_{0}\), \(\omega \), and \( \tau \) and $\displaystyle \psi(t) = \tau(t)+\int_{0}^{t} \frac{\tau(s)}{s}ds$.
    
\end{theorem}

\begin{remark}
In our approach to proving Theorem \ref{MainThm01}, the assumptions $\rm{(A3)}$–$\rm{(A4)}$, namely $f$, $\theta_{F}$, and $\mathfrak{B}$ share the same Dini modulus of continuity, are not restrictive in the following sense: if $f$, $\theta_{F}$, and $\mathfrak{B}$ are Dini continuous in the $L^{p}$-sense  (for some $p \geq n$) with moduli of continuity $\zeta_{f}$, $\zeta_{\theta_{F}}$, and $\zeta_{\mathfrak{B}}$, respectively, and each modulus satisfies condition $\rm{(A4)}$, then one can take the modulus of continuity $\varpi(t) = \max\{\zeta_{f}(t), \zeta_{\theta_{F}}(t), \zeta_{\mathfrak{B}}(t)\}$, which also satisfies the conditions $\rm{(A3)}-\rm{(A4)}$. 

Therefore, in this scenario, we obtain for flat viscosity solutions to \eqref{MainEq} that \( u \in C^{2,\hat{\psi}}(\mathrm{B}_{1/2}) \) and
\[
\|u\|_{C^{2,\hat{\psi}}(\mathrm{B}_{1/2})} \leq \mathrm{C}_0 \cdot \delta_0,
\]
where \( \mathrm{C}_0>0\) depends only \( n\), \(\lambda\), \(\Lambda\),  \(\alpha_{0}\), \(\omega \), \(\zeta_{f}\), \(\zeta_{\theta_{F}}\), \(\zeta_{\mathfrak{B}}\), and $\displaystyle \hat{\psi}(t) = \varpi(t)+\int_{0}^{t} \frac{\varpi(s)}{s}ds$.
\end{remark}

\subsection*{Some applications and consequences}

As an immediate consequence of our results, we recover the ones established by dos Prazeres and Teixeira in~\cite[Theorem 2.2]{dosPrazTei2016} (without a linear drift term), which considers the case of H\"{o}lder continuous data and an identically zero drift term. 

An insightful interpretation of Theorem \ref{MainThm01} is that if a function $u$ solves a fully nonlinear elliptic equation with $C^{0, \text{Dini}}$ coefficients and source term, then—provided that $u$ is sufficiently close to a $C^{2,\psi}$ function—it must belong to the class $C^{2,\psi}$. This perspective is particularly relevant in problems involving certain \textit{a priori} prescribed data.

As an application of our findings, we addressed an Evans-Krylov-type result for classical solutions to \eqref{MainEq}.

\begin{corollary}[{\bf $C^2$ implies $C^{2,\text{Dini}}$}]\label{Corolario-EK}
Let $u \in C^2(\mathrm{B}_1) \cap \mathscr{F}$ be a classical solution. Then, $u \in C^{2,\psi}(\mathrm{B}_{1/2})$, and there exists a constant $\mathrm{C}$ depending only on $n$, $p_{0}$, $\lambda$, $\Lambda$, $\omega$, $\tau(1)$ and  $\|u\|_{C^2(\mathrm{B}_1)}$ such that
\[
\|u\|_{C^{2,\psi}(\mathrm{B}_{1/2})} \leq \mathrm{C}(n,p_{0}, \lambda, \Lambda, \omega, \tau(1), \|u\|_{C^2(\mathrm{B}_1)}).
\]
\end{corollary}

\bigskip
We must recall that pointwise regularity estimates have significant applications in other areas such as free boundary problems \cite{Lian-Zhang2023}, \cite{Lind-Monn2013}, \cite{Lind-Monn2015}, \cite{Monneau2009} and, \cite[Chapter 7]{PetSHaUral12} (for instance, in the study of the structure of singular sets in obstacle-type problems), and in the analysis of the structure of nodal sets of certain uniformly elliptic PDEs, see Han's work \cite[Theorem 5.1]{Han2000} for an insightful result. We recommend to the reader \cite{Lian-Wang-Zhang2024} an enlightening survey on pointwise estimates for fully nonlinear elliptic PDEs.

In this direction, we may consider 
the nodal sets of solutions:
\[
\mathcal{N}_k(u, \overline{\mathrm{B}}_{1/2}) := \left\{ x_0 \in \overline{\mathrm{B}}_{1/2} : u(x_0) = |D u(x_0)| = \cdots = |D^{k-1}u(x_0)| = 0, \; |D^k u(x_0)| \neq 0 \right\}.
\]
Then, we have the following characterization of nodal sets:

\begin{corollary}\label{T3}
Consider \(u \in \mathscr{F}\) such that $u$ is a flat viscosity solution. 
Then, 
\[
\mathcal{N}_{1}(u,\overline{\mathrm{B}}_{1/2})=\bigcup_{i=0}^{n}\mathcal{M}_{1}^{i}\,\,\,\text{and}\,\,\, \mathcal{N}_{2}(u,\overline{\mathrm{B}}_{1/2})=\bigcup_{i=0}^{n-1}\mathcal{M}_{2}^{i}
\]
where each \(\mathcal{M}_{j}^{i}
\) is on a finite union of \(i\)-dimensional \(C^{1,\psi}\) -  manifolds, where \(j=1,2\).
\end{corollary}

The proof of this result proceeds along the same lines as Han's original argument in \cite[Theorem~5.1]{Han2000} (see also \cite[Theorem~1.26]{Lian2024} for the parabolic analogue):
\begin{itemize}
  \item[\checkmark] First, we derive higher-order pointwise expansions at nodal points through Schauder-type estimates (using Theorem~\ref{MainThm01});
  \item[\checkmark] Next, we employ Whitney’s Extension Theorem (cf.~\cite{Fefferman2009}) to extend $u$ with the required regularity;
  \item[\checkmark] Finally, we characterize the nodal set as the zero level set of a suitable vector-valued map constructed from derivatives of $u$, invoking the Implicit Function Theorem (see \cite{Fefferman2009}) to complete the argument.
\end{itemize}
 For this reason, we have decided to omit the proof.

\medskip

Finally, motivated by the study of the fully nonlinear Alt–Phillips equation, Wu and Yu in \cite{WH2022} assumed, in addition to the ellipticity condition $(\mathrm{A1})$, that the operator \( F: \text{Sym}(n) \to \mathbb{R} \) satisfies the following conditions:
\begin{equation}\label{Condi-WuYu}\tag{SC}
\left\{
\begin{array}{l}
F \text{ is convex,} \\
F(\mathbf{O}_n) = 0, \\
\text{the trace operator is a sub-differential of } F \text{ at } \mathbf{O}_n.
\end{array}
\right.
\end{equation}

Given a matrix \( \mathfrak{A} \in \text{Sym}(n) \), a \textbf{sub-differential} of \( F \) at \( \mathfrak{A} \) is a linear operator \( \mathcal{S}_{\mathfrak{A}} : \text{Sym}(n) \to \mathbb{R} \) such that
\[
\mathcal{S}_{\mathfrak{A}}(\mathrm{X}) \leq F(\mathfrak{A} + \mathrm{X}) - F(\mathfrak{A}) \quad \text{for all } \mathrm{X} \in \text{Sym}(n).
\]

Furthermore, observe that the ellipticity condition $(\mathrm{A1})$ implies that \( \mathcal{S}_{\mathfrak{A}} \) constitutes a uniformly elliptic operator.

Additionally, Wu and Yu in \cite{WH2022} also assumed that either
\begin{itemize}
\item[$(\mathrm{F}1)$] $F$ is differentiable at $\mathbf{O}_n$, or
\item[$(\mathrm{F}2)$] $F$ is $1$-homogeneous, i.e., $F(\mu \mathrm{M}) = \mu F(\mathrm{M})$ for all $\mu > 0$ and $\mathrm{M} \in \text{Sym}(n)$.
\end{itemize}

Under these assumptions, in conjunction with condition \eqref{Condi-WuYu}, it follows that
\begin{equation}\label{Eq-Linearization}
\mathrm{D}_{\mathrm{M}}F(\mathbf{O}_n) = \mathrm{Tr}(\mathrm{M}).    
\end{equation}

For our regularity context, these conditions are crucial for deriving Schauder estimates for flat viscosity solutions and align with our present findings. Indeed, we can obtain the following Schauder estimates (cf. \cite{Bao2002}, \cite[§2]{Kovats99}, \cite[§3]{Kovats99} and \cite[Theorem 1.1]{Wang2006}):

\begin{theorem}[{\bf Local $C^{2,\textrm{Dini}}$ regularity - Convex case}]\label{MainThm02}
Let \( u \in C^0(\mathrm{B}_1) \) be a viscosity solution to \eqref{MainEq}, where the assumptions $(\mathrm{A1}), (\mathrm{A3}), (\mathrm{A4})$, \eqref{Condi-WuYu} and $(\mathrm{F}1)$ or $(\mathrm{F}2)$ are in force. There exists \( \delta_0 > 0 \), depending only upon \( n\), \(\lambda\), \(\Lambda\), and \( \tau(1) \), such that if
\[
\|u\|_{L^{\infty}(\mathrm{B}_1)} \leq \delta_0,
\]
then \( u \in C_{\text{loc}}^{2,\psi}(\mathrm{B}_{1}) \). Moreover, the following estimate holds
\[
\|u\|_{C^{2,\psi}(\mathrm{B}_{1/2})} \leq \mathrm{C}_0 \cdot \delta_0,
\]
where \( \mathrm{C}_0>0\) depends only upon \( n\), \(\lambda\), \(\Lambda\), \(\alpha_{0}\), and \( \tau \) and $\displaystyle \psi(t) = \tau(t)+\int_{0}^{t} \frac{\tau(s)}{s}ds$.
\end{theorem}

\begin{remark} For clarity, we emphasize that the derivation of the preceding Schauder estimates relies decisively on the $F$--harmonic approximation established in Lemma~\ref{aproxlemma}. Under the assumptions of Theorem~\ref{MainThm02}, the geometric tangential equation arising in the \textit{Reductio ad Absurdum} argument of Lemma~\ref{aproxlemma} reduces precisely to $\Delta \mathfrak{h} = 0$. Consequently, the geometric iteration scheme developed in Lemma~\ref{approxlemmadelta} and Proposition~\ref{Prop5.1} follows directly, yielding the desired Schauder estimates for flat viscosity solutions in the convex setting.

\end{remark}

\medskip

Finally, we recall that Kovats' seminal work \cite{Kovats97} established Schauder estimates under the convexity hypothesis on $F$ (see also the state-of-the-art discussion in Section~\ref{State-of-the-art}). Nonetheless, Theorem~\ref{MainThm02} sharpens Kovats' result in the context of flat viscosity solutions with convex nonlinearity. Specifically, the right-hand side of our estimates is independent of the norms of solution and the associated data, unlike in Kovats' framework. In addition, our model PDE incorporates drift terms, which were not addressed in the original analysis of \cite{Kovats97}.

\subsection*{The \textit{strategi} behind the geometric tangential approach}

We conclude this introduction by outlining the heuristic framework of the geometric tangential analysis underlying our proofs. By way of explanation, consider a fully nonlinear elliptic operator \( F: \text{Sym}(n) \to \mathbb{R} \) such that $F(\mathbf{O}_n) = 0$ (this is not a restrictive assumption). Then, the family of elliptic scaling functions defined by
\[
\mathcal{G}_{\sigma}(\mathrm{X}) := \frac{1}{\sigma} F(\sigma \mathrm{X}), \quad \text{for } \,\,\,\sigma > 0.
\]
constitutes a continuous curve of operators that preserves the ellipticity constants of the original equation (i.e., $0< \lambda\leq \Lambda$). Indeed, $\mathcal{G}_{\sigma}$ satisfies the same assumptions $(\mathrm{A}1)-(\mathrm{A}4)$.

Now, if \( F: \text{Sym}(n) \to \mathbb{R}\) is differentiable (for instance, at the origin, recalling the normalization \( F(\mathbf{O}_n) = 0 \)), then we have
\[
\displaystyle \lim_{\sigma \to 0} \mathcal{G}_{\sigma}(\mathrm{X}) = \lim_{\varrho \to 0} \frac{F(\sigma \mathrm{X}+\mathbf{O}_n)-F(\mathbf{O}_n)}{\sigma} = \mathrm{D}_{\mathrm{X}}F(\mathbf{O}_n) = \sum_{i,j=1}^{n}\frac{ \partial F}{\partial \mathrm{X}_{ij}}(\mathbf{O}_n) \mathrm{X}_{ij}.
\]
In other words, the linear, uniformly elliptic operator \(\displaystyle \mathrm{X} \mapsto \sum_{i,j=1}^{n}\frac{ \partial F}{\partial \mathrm{X}_{ij}}(\mathbf{O}_n) \mathrm{X}_{ij}  = \mathrm{Tr}(\mathfrak{A}_0\mathrm{X})\) represents the tangential equation associated with \(\mathcal{G}_{\sigma} \) in the limiting configuration as \( \sigma \to 0 \), where $(\mathfrak{A}_0)_{ij} = \frac{ \partial F}{\partial \mathrm{X}_{ij}}(\mathbf{O}_n)$ for every $1 \leq i, j\leq n$. 

Now, if \( u: \mathrm{B}_1 \to \mathbb{R} \) is a solution to an equation involving the original operator \( F: \text{Sym}(n) \to \mathbb{R}\), i.e.,
$$
F(D^2 u) + \langle \mathfrak{B}(x), Du\rangle = f(x) \quad \text{in} \quad \mathrm{B}_1,
$$
then the scaled function \( u^{\mathfrak{a}}_{\sigma}(x) := \frac{1}{\sigma^{2}\tau(\sigma)\mathfrak{a}} u(\sigma x) \) (a scaled and normalized profile) satisfies a corresponding equation for 
$$
\mathcal{G}_{\sigma^{\alpha} \mathfrak{a}}(D^2 u^{\mathfrak{a}}_{\sigma}(x)) + \langle \mathfrak{B}_{\sigma}^{\mathfrak{a}}(x), D u^{\mathfrak{a}}_{\sigma}(x)\rangle = \frac{1}{\tau(\sigma) \mathfrak{a}} f(\sigma x) := f_{\sigma}^{\mathfrak{a}}(x) \quad \text{in} \quad \mathrm{B}_{1}, 
$$
where $\mathfrak{B}_{\sigma}^{\mathfrak{a}}(x) = \frac{\sigma}{\mathfrak{a}} \mathfrak{B}(\sigma x)$, and  $\mathfrak{a}>0$ is chosen in such a way $\|f_{\sigma}^{\mathfrak{a}}\|_{\infty, \mathrm{B}_1} = \text{o}(1)$ and $\|\mathfrak{B}_{\sigma}^{\mathfrak{a}}\|_{\infty, \mathrm{B}_1} = \text{o}(1)$ as $\sigma \to 0$.

This strategy allows us to access some results on the classification of global profiles (i.e., Liouville-type results) available for the linear tangential equation via compactness and stability arguments. Indeed, we can conclude the following:
$$
\begin{array}{ccc}
 \left\{ 
 \begin{array}{lcl}
 \mathcal{G}_{\sigma^{\alpha} \mathfrak{a}}(\mathrm{X}) \to  \mathrm{Tr}(\mathfrak{A}_0 \mathrm{X}) & \text{as}& \sigma \to 0    \\
 \mathfrak{B}_{\sigma}^{\mathfrak{a}} \to 0 & \text{as} & \sigma \to 0\\
 f_{\sigma}^{\mathfrak{a}} \to 0 & \text{as}& \sigma \to 0\\
 u^{\mathfrak{a}}_{\sigma} \to \mathbf{U} & \text{as}& \sigma \to 0
 \end{array}
 \right.& \Rightarrow & \mathrm{Tr}(\mathfrak{A}_0 D^2 \mathbf{U}) = 0 \quad \text{in} \quad \mathrm{B}_{3/4} \quad \Rightarrow \quad \mathbf{U} \in C_{\text{loc}}^{2,1} 
\end{array}
$$

Hence, we transfer these sharp limiting estimates to $u^{\mathfrak{a}}_{\sigma}$, appropriately adjusted via the geometric tangential path employed to access the tangential linear elliptic regularity theory. It relies on a control of oscillation decay obtained from the regularity
theory available for a fine limiting equation (a harmonic profile); the realm of the so-called \textbf{geometric tangential
analysis}. Precisely, by performing a geometric iteration, we establish pointwise $\mathrm{C}^{2,\psi}$ regularity for solutions to the problem \eqref{MainEq}: there exists a sequence of quadratic polynomials $\{\mathfrak{P}_k\}_{k \in \mathbb{N}}$ such that
    $$
 \displaystyle\sup_{\mathrm{B}_{\rho^k}(x_0)} \frac{\left|u(x)-\mathfrak{P}_k(x)\right|}{\rho^{2k}\tau(\rho^k)}\leq 1 \quad \stackrel[\Longrightarrow]{\text{Geometric }}{{\text{iteration}}} \,\,\,u \,\, \text{is} \,\,\,\mathrm{C}^{2, \psi} \quad \text{at}\,\,\,x_0,\,\,\,\text{where}\,\,\,\psi(t) = \tau(t) + \int_{0}^{t} \frac{\tau(s)}{s}ds.
$$

Finally, a normalization, scaling process, and a standard covering argument reduce the general setting to the analysis of the previous steps.

We recommend to the reader the following enlightening recent surveys \cite[Ch. 5]{João},  \cite{Teixeira2016}, and \cite{Teixeira2020} for the development of lines of investigation on geometric tangential methods and their application in nonlinear PDEs and related topics.

\section{Schauder estimates: a brief mathematical tour}\label{State-of-the-art}

The development of Schauder theory has significantly shaped the modern perspective that solving a PDE is, in essence, equivalent to establishing an \textbf{a priori} estimate, i.e., obtaining bounds on a solution before constructing it explicitly. The Schauder estimates play a fundamental role in the theory of linear/nonlinear elliptic PDEs due to a wide class of applications and connections with other areas of mathematics (see \cite{Kichenassamy2006} for an enlightening survey on Schauder estimates and their several applications).

Throughout the last decades, several proofs were developed, among which we highlight the following strategies:

\begin{itemize}
    \item[\checkmark] Potential-theoretic approaches based on the fundamental solution, which dates back to Schauder's original work in the 1930s (see \cite[Ch. 4.3]{Friedman1964} and \cite[Ch. 4]{Gilbarg-Tru2001} and \cite{Kichenassamy2006});
    \item[\checkmark] Energy considerations by invoking Caccioppoli-type inequalities (see \cite[Ch. 3.4]{João} and \cite[Ch. 3.4]{Han-Lin-Book});
    \item[\checkmark] An approach based on Maximum/Comparison Principle (see the books \cite[Ch. 3.3]{João}, \cite[Theorem 2.20]{FR-O}, and the manuscript \cite[Theorem 1.1]{Wang2006}); 
    \item[\checkmark] Blow-up arguments applied by Simon in the seminal work \cite{Sim97} (cf.  \cite[Theorem 2.20]{FR-O});
    \item[\checkmark] Convolution operator techniques employed by Peetre in \cite{Peetre1966};
    \item[\checkmark] Mollification methods  utilized by Trudinger in \cite{Trud1986};
    \item[\checkmark] Perturbation methods developed by Safonov in \cite{Safonov1984} and \cite{Safonov1989}, and Caffarelli in \cite{Caffarelli1989} (see also \cite{CafCabre1995}).
\end{itemize}

Another powerful technique is based on the polynomial characterization of $C^{k,\alpha}$ spaces, introduced by Campanato in \cite{Campanato1963} and \cite{Campanato1964} (see Kovats' work \cite{Kovats99} for a characterization of Dini-Campanato spaces). In this context, Caffarelli
and Safonov (see above references) 
were among the first authors to use this characterization to establish pointwise Schauder estimates.

\bigskip

In the scenario of linear equations in non-divergence form, we have the following version of Schauder estimates: 

\begin{theorem}[{\bf Interior estimates - \cite[
Theorem 2.20]{FR-O}}]\label{thm:schauder}
Let $\alpha \in (0,1)$ fixed, and $u \in C^{2,\alpha}(\mathrm{B}_1)$ be any solution to
\[
\mathrm{Tr}(\mathfrak{A}(x)D^2u(x)) \defeq \sum_{i,j=1}^{n} a_{ij}(x) \, \partial_{ij} u(x) = f(x) \quad \text{in } \mathrm{B}_1,
\]
with $f \in C^{0,\alpha}(\mathrm{B}_1)$ and $a_{ij} \in C^{0,\alpha}(\mathrm{B}_1)$, where the matrix $\mathfrak{A}(x) = (a_{ij}(x))_{i,j}$ satisfies the uniform ellipticity condition
\begin{equation}\label{eq:ellipticity}
\lambda|\xi|^2 \leq \sum_{i,j=1}^{n} a_{ij}(x)\xi_i\xi_j \leq \Lambda|\xi|^2 \quad \text{for all } x \in\mathrm{B}_1\,\,\,\text{and}\,\,\, \xi \in \mathbb{R}^n
\end{equation}
for some constants $0 < \lambda \leq \Lambda < \infty$. Then,
\[
\|u\|_{C^{2,\alpha}(\mathrm{B}_{1/2})} \leq \mathrm{C}\left( \|u\|_{L^{\infty}(\mathrm{B}_1)} + \|f\|_{C^{0,\alpha}(\mathrm{B}_1)} \right)
\]
for some constant $\mathrm{C} > 0$ depending only on $\alpha$, $n$, $\lambda$, $\Lambda$, and $\|a_{ij}\|_{C^{0,\alpha}(\mathrm{B}_1)}$.
\end{theorem}

\medskip

In the linear case, see \cite[§2]{Kovats99} for further details, when considering solutions $u \in C^2(\mathrm{B}_1)$ of the Poisson equation
\[
\Delta u(x) = f(x) \quad \text{in} \quad \mathrm{B}_1,
\]
with $f$ a Dini continuous function, the modulus of continuity of $D^2u$ is well understood. In particular, one has the estimate
\[
\sup_{|x - y| \leq t} |D^2u(x) - D^2u(y)| \leq \mathrm{C}_0 \left( \int_0^t \frac{\omega(r)}{r} \, dr + t \int_t^c \frac{\omega(r)}{r^2} \, dr \right),
\]
for some constant $\mathrm{C}_0 > 0$.

\medskip
In the manuscript \cite[Theorem 1.1]{Wang2006},  Wang addressed an elementary and simple proof for the Schauder estimates for elliptic equations. Precisely, let $u \in C^2$ be a solution of $\Delta u(x) = f(x)$ in $\mathrm{B}_1$. Then, for all $x, y \in \mathrm{B}_{1/2}$,
\begin{equation}
|D^2 u(x) - D^2 u(y)| \leq \mathrm{C}_n \left[ d \|u\|_{L^{\infty}(\mathrm{B}_1)} + \int_0^d \frac{\omega(r)}{r} \, dr + d \int_d^1 \frac{\omega(r)}{r^2} \, dr \right],
\end{equation}
where $d = |x - y|$ and $\mathrm{C}_n > 0$ is a constant depending only on $n$. Moreover, Wang's approach also applies to fully nonlinear  PDEs, see \cite[Theorem 3.1]{Wang2006}.

\medskip
In the fully nonlinear setting, Evans \cite{Evans1982} and Krylov  \cite{Krylov1983} proved independently that if $u \in C^2(\mathrm{B}_1)$ is a solution of the nonlinear equation $F(D^2 u) = 0$ in $\mathrm{B}_1$, then $D^2u$ satisfies a H\"{o}lder continuity estimate of the form
\[
|D^2u(x) - D^2u(y)| \leq \mathrm{C}|x - y|^{\alpha_{\mathrm{EK}}} \quad \text{in } \mathrm{B}_{1/2},
\]
for some constant $\alpha_{\mathrm{EK}} \in (0,1)$ and $\mathrm{C}>0$ depending only on ellipticity constants and dimension. 

\medskip

Subsequently, Caffarelli established Schauder estimates for the inhomogeneous problem via a perturbation and compactness technique (see \cite[Ch. 8]{CafCabre1995}). Precisely, if \( u \in C^0(\mathrm{B}_1) \) is a viscosity solution to
\[
F(D^2 u) = f(x) \quad \text{in } \mathrm{B}_1 \subset \mathbb{R}^n, \quad f \in C^{0, \alpha}(\mathrm{B}_1) \quad  \text{for some}\,\,\,  \alpha \in (0,1) \quad \text{and}\quad F \,\,\,\text{convex},
\]
then, there exists a constant \( \alpha_0 \in (0,1) \), depending only on \( n, \lambda \) and \( \Lambda \), such that \( u \in C^{2,\min\{\alpha_0, \alpha\}}(\mathrm{B}_{1/2}) \), and
    \[
    \|u\|_{C^{2,\min\{\alpha_0,\alpha\}}(\mathrm{B}_{1/2})} \leq \mathrm{C}\left( \|u\|_{L^{\infty}(\mathrm{B}_1)} + \|f\|_{C^{0, \alpha}(\mathrm{B}_1)} \right),
    \]
    for some constant \( \mathrm{C} > 0 \) depending only on \( n, \lambda \), \( \Lambda \), and \( \alpha \). 
    
\medskip

Kovats in \cite{Kovats97} investigates the interior regularity and local solvability of concave, fully nonlinear, uniformly elliptic equations of the form
\[
F(D^2u) = f \quad \text{in } \Omega.
\]
The primary objective is to analyze the modulus of continuity of $D^2u$ in terms of the modulus of continuity of $f$, which is assumed to be Dini continuous, in other words, for
 $ \displaystyle \omega(t) = \sup_{x, y \in \Omega \atop{|x - y|} \leq t} |f(x) - f(y)|$, one requires $\displaystyle
\mathfrak{I}_{\omega} \defeq \int_0^t \frac{\omega(r)}{r} \, dr < +\infty.$

In this context, Kovats considers solutions $u \in C^2(\mathrm{B}_1)$ of the nonlinear equation with $f \not\equiv 0$, where $f$ is a Dini continuous function in $\mathrm{B}_1$ that satisfies a certain technical condition on $\omega(t)$. It is established that $D^2u$ has a modulus of continuity bounded by
\[
\sup_{|x - y| \leq t} |D^2u(x) - D^2u(y)| \leq \mathrm{C}^{\ast}(\textit{universal}) \left( \int_0^t \frac{\omega(r)}{r} \, dr + t^{\alpha_{\mathrm{EK}}} \right),
\]
where $\alpha_{\mathrm{EK}} \in (0, 1)$ is the Evans-Krylov exponent. In deriving these results, in addition to employing the Evans-Krylov theory, the author utilizes techniques introduced by Safonov and Caffarelli, including polynomial approximation and Maximum Principle methods.

\medskip

Subsequently, in the work \cite[§3]{Kovats99}, Kovats employed a generalized version of Campanato spaces (namely, Dini-Campanato spaces) to establish a regularity result for a broad class of fully nonlinear elliptic equations of the form
\[
F(D^2u, x) = f(x).
\]
The operator $F$ is assumed to be concave, and the function $f$ is assumed to be Dini-continuous in the $L^n$-norm. Under these assumptions, estimates for the modulus of continuity of $D^2u$ are derived.

\medskip
We must highlight that Bao in \cite{Bao2002} establishes existence and regularity results for viscosity solutions of fully nonlinear second-order uniformly elliptic equations subject to Dirichlet boundary conditions on general bounded domains $\Omega \subset \mathbb{R}^N$. For irregular domains $\Omega$, the boundary condition $u = 0 \quad \text{on } \partial\Omega$ is interpreted in a weak sense through the use of a barrier function associated with the operator. This approach was originally introduced for linear elliptic equations in \cite{BNV1994}. The existence of solutions is obtained via an approximation procedure, which involves penalizing both the equation and the domain. Interior regularity of viscosity solutions in the $C^{2,\psi}$ class, where $\psi$ is a modulus of continuity, is derived using Caffarelli's perturbation technique (see \cite{Caffarelli1989}).

\medskip

Concerning recent trends on (local) higher regularity estimates, in his seminal paper \cite{Savin2007}, Savin investigates viscosity solutions of general second-order fully nonlinear equations of the form
\[
F(D^2u, Du, u, x) = 0,
\]
for which $u = 0$ is a solution. Assuming that $F$ is smooth and uniformly elliptic only in a neighborhood of the points $(0, 0, 0, x)$, Savin establishes interior $C^{2,\alpha}$ regularity for \textbf{flat solutions}, namely, solutions $u$ with sufficiently small $L^{\infty}$-norm. More precisely, the operator $F$ is initially assumed to be merely measurable, and a Harnack inequality is derived for flat solutions; higher regularity is then obtained under the additional assumption that $F$ is smooth. 
\medskip

Cabr\'{e} and Caffarelli \cite[Theorem 1.1]{CabreCaff2003} examined alternative conditions on $F$, namely, weaker convexity/concavity assumptions,
 that guarantee classical solutions to $F(D^2u) = 0$. Their work establishes local $C^{2,\alpha}$ regularity estimates for a class of non-convex/concave operators.

Subsequently, dos Prazeres-Teixeira in the magnum opus manuscript \cite{dosPrazTei2016} studied equations of the form
\begin{equation}\label{Eq-PrazTei}
F(D^2u,x) = f(x)
\end{equation}
where $F$ is a non-convex, fully nonlinear, uniformly elliptic operator. They establish local $C^{2,\alpha}$ regularity for flat solutions of \eqref{Eq-PrazTei},  
under the assumption that the coefficients of $F$ and the source term $f$ belong to the class $C^{0,\alpha}$. The proofs of these results rely on a combination of geometric tangential analysis and perturbative arguments inspired by compactness techniques in the contemporary theory of elliptic PDEs. The authors construct a family of elliptic rescalings that preserve the ellipticity constants of the original operator, thereby enabling the application of tangential linear elliptic regularity theory.

Cao \textit{et al.} \cite{CLW2011} obtained local $C^{2,\alpha}$ estimates for classical solutions of fully nonlinear, uniformly elliptic equations $F(D^2u) = 0$, expressed in terms of the Hessian matrix $D^2u$'s modulus of continuity. Their analysis assumes $F$ is merely locally $C^{1,\beta}$, dispensing with convexity or concavity requirements. Through an iterative scheme based on $L^2$ estimates, they derived H\"older continuity for the Hessian.

Recently, Bhattacharya and Warren, in \cite{BhatWarr2021}, established explicit interior \( C^{2,\alpha} \) estimates for viscosity solutions of fully nonlinear, uniformly elliptic equations that are sufficiently close to linear equations. Moreover, they provided a precise quantitative bound characterizing this closeness.

\begin{definition}[{\cite[
Definition 1.2]{BhatWarr2021}}] A uniformly elliptic, non-linear operator $F$ is said \textbf{almost linear} with constant $\varepsilon>0$ if
\begin{equation}
\|\mathrm{D}_{\mathrm{X}}F(\mathrm{M}) - \mathrm{D}_{\mathrm{X}}F(\mathrm{N})\| \leq \varepsilon
\end{equation}
for all $\mathrm{M}, \mathrm{N} \in \text{Sym}(n)$. Moreover, one defines $\varepsilon$ to be the \textbf{closeness constant} of $F$.
    
\end{definition}

\begin{theorem}[{\cite[
Theorem 1.4]{BhatWarr2021}}] Given $0<\lambda\leq \Lambda$, and $0 < \alpha < \overline{\alpha} < 1$, suppose that $F$ is almost linear with constant $\varepsilon_0$, and let $u \in C^0(\mathrm{B}_1)$ be a viscosity solution of 
$$
F(D^2u) = f(x) \quad \text{in} \quad \mathrm{B}_1. 
$$
If $f \in C^{0,\alpha}(\mathrm{B}_1)$, then $u \in C^{2,\alpha}(\mathrm{B}_{1/2})$ and the following estimate holds:
\begin{equation}
\|u\|_{C^{2,\alpha}(\mathrm{B}_{1/2})} \leq \mathrm{C}\left( \|u\|_{L^{\infty}(\mathrm{B}_1)} + \|f\|_{C^{0,\alpha}(\mathrm{B}_1)} \right),
\end{equation}
where $\mathrm{C}>0$ depends only on $n$, $\lambda$, $\Lambda$, $\varepsilon_0$, $\alpha$, and $\overline{\alpha}$.
\end{theorem}

\medskip 

Related results involving Cordes-type ellipticity conditions were also obtained by Huang in \cite{HuangAIHP}, where Morrey-type and BMO estimates are proved under a quantitative closeness assumption. In addition, Huang \cite{HuangPAMS} established Schauder-type estimates for fully nonlinear uniformly elliptic equations satisfying a polynomial Liouville property, which yield $C^{2,\alpha}$ regularity for every $\alpha<1$ without requiring concavity or convexity of $F$.

Another condition that can lead to improved regularity arises when the ellipticity constants \( 0 < \lambda \leq \Lambda \) are sufficiently close. Specifically, Wu and Niu, in \cite[Theorem 1.4]{WuNiu2023}, employ compactness techniques to establish interior \( C^{2,\alpha} \) regularity for viscosity solutions of fully nonlinear, uniformly elliptic equations under the assumption that the ellipticity constants are nearly equal. It is worth noting that the homogeneous version of this result had already been obtained nearly a decade earlier by Da Silva in his Ph.D. thesis (see \cite[Ch. 5]{DaSilva-PhDThesis}).

\medskip
In conclusion, Goffi, in \cite{Goffi2024}, undertakes a comprehensive study of \textit{a priori} estimates and Evans-Krylov-type regularity results in H\"{o}lder spaces for fully nonlinear, second-order, uniformly elliptic or uniformly parabolic equations, even in the absence of concavity or convexity of the operator on the space of symmetric matrices. Furthermore, Goffi addresses the following generalization of Nirenberg’s classical result \cite{Nirenberg1953} to the setting of viscosity solutions.

\begin{theorem}[{\bf \cite[Theorem Appendix A.1]{Goffi2024}}]
    
Let $u : \mathrm{B}_1 \to \mathbb{R}$, with $\mathrm{B}_1 \subset \mathbb{R}^2$, and suppose that $u$ is a continuous viscosity solution to
\[
F(D^2 u) = f(x) \quad \text{in } \mathrm{B}_1.
\]
Assume that $F: \text{Sym}(2) \to \mathbb{R}$ is uniformly elliptic (no other assumptions are required), and that $f \in C^{0,\alpha}(\mathrm{B}_1)$. Then, for some small $\alpha \in (0, 1)$, we have the regularity estimate
\[
\|u\|_{C^{2,\alpha}(\mathrm{B}_{1/2})} \leq \mathrm{C}\left( \|u\|_{L^{\infty}(\mathrm{B}_1)} + \|f\|_{C^{0, \alpha}(\mathrm{B}_1)} \right),
\]
where the constants $\alpha$ and $\mathrm{C}>0$ depend only on $\lambda$ and $\Lambda$.
\end{theorem}

Therefore, the preceding (albeit incomplete) literature review served as our primary motivation to derive local Schauder-type estimates for flat solutions of non-convex, fully nonlinear elliptic PDEs with a drift term, under suitable Dini continuity assumptions on the data. Our approach relies on geometric tangential methods, compactness arguments, and approximation schemes. Furthermore, to the best of our knowledge, these results are new even in the case of the model problem with vanishing drift.

\section{Preliminaries}

In this section, we provide essential definitions and auxiliary results that are fundamental to our approach in establishing the Schauder estimates, within the context of our discussion.

First, we recall the definition of Dini spaces.

\begin{definition}[{\bf The Space \(C^{k,\mathrm{Dini}}(\Omega)\)} - {\cite{Kovats97}\cite{João}}]

Let \( f: \Omega \subset \mathbb{R}^n \to \mathbb{R} \) be a function of class \( C^k \). We say that \( f \in C^{k,\mathrm{Dini}}(\Omega) \) if all partial derivatives of order \(k\) are Dini-continuous; that is, for every multi-index \( \beta \in \mathbb{N}^n \) with \( |\beta| = k \),
\[
[D^\beta f]_{\mathrm{Dini}(\Omega)} := \int_0^1 \frac{\omega_{D^\beta f}(r)}{r} \, dr < \infty,
\]
where the \textbf{modulus of continuity} \( \omega_{D^\beta f} \) is given by
\[
\omega_{D^\beta f}(r) := \sup_{\substack{x,y \in \Omega \\ |x - y| \leq r}} |D^\beta f(x) - D^\beta f(y)|.
\]

\medskip

In this case, we define the \textbf{norm} in the space \(C^{k,\mathrm{Dini}}(\Omega)\) as:
\[
\|f\|_{C^{k,\mathrm{Dini}}(\Omega)} := \sum_{|\beta| \leq k} \|D^\beta f\|_{L^\infty(\Omega)} + \sum_{|\beta| = k} [D^\beta f]_{\mathrm{Dini}(\Omega)}.
\]

\medskip

The quantity
\[
[f]_{\mathrm{Dini}(\Omega)} := \int_0^1 \frac{\omega_f(r)}{r} \, dr
\]
is called the \textbf{Dini semi-norm}, which provides a weaker regularity than the classical H\"{o}lder semi-norm:
\[
[f]_{C^{0,\beta}(\Omega)} := \sup_{x, y \in \Omega \atop{x \neq y}} \frac{|f(x) - f(y)|}{|x - y|^\beta}.
\]
\end{definition}

The following definition introduces viscosity solutions—a robust notion suited for situations where classical differentiability fails. This framework is essential for studying fully nonlinear elliptic equations, especially when classical solutions are unavailable or hard to obtain.
\begin{definition}[{\bf Viscosity Solutions \cite[Remark 2.2]{JVGN21}}]
A continuous function $u \in C^0(\mathrm{B}_1)$ is said to be a \textit{viscosity subsolution} to \eqref{MainEq} in $\mathrm{B}_1$ if for any $x_{0}\in\mathrm{B}_{1}$ and $\varphi\in C^{2}(\mathrm{B}_{1})$ such that $u - \varphi$ has a local maximum at $x_0$, the following inequality holds:
\[
F( D^2 \varphi(x_0),x_0) + \langle \mathfrak{B}(x_0), D\varphi(x_{0})\rangle \leq f(x_0).
\]
Similarly, $u$ is called a \textit{viscosity supersolution} to \eqref{MainEq} in $\mathrm{B}_1$ if for any $x_{0}\in \mathrm{B}_{1}$ and $\varphi\in C^{2}(\mathrm{B}_{1})$ such that $u - \varphi$ has a local minimum at $x_{0}$, one has:
\[
F(D^2 \varphi(x_0),x_0) + \langle \mathfrak{B}(x_0), D\varphi(x_{0})\rangle \geq f(x_0).
\]
We say that $u$ is a \textit{viscosity solution} to \eqref{MainEq} if it is both a viscosity subsolution and a viscosity supersolution.
\end{definition}
We now state an interior regularity result, which provides a H\"{o}lder estimate for the gradient of viscosity solutions under suitable structural and integrability conditions.
\medskip
\begin{theorem}[{\bf H\"{o}lder Gradient Estimate \cite[Theorem 1.6]{JVGN21}}]\label{C1alpharegularity}
Assume \textnormal{(A1)}–\textnormal{(A3)}. Let $f \in L^p(\Omega)\cap C^0(\Omega)$, where $p > n$, and $\Omega \subset \mathbb{R}^n$ be a bounded domain. Let $u$ be a bounded viscosity solution of \eqref{MainEq}. Then, there exists $\alpha \in (0,1)$ and $\theta = \theta(\alpha)$, depending on $n, p, \lambda, \Lambda$, and $\|\mathfrak{B}\|_{L^{\infty}(\mathrm{B}_1)}$ such that if 
$$
\displaystyle \left(\intav{\mathrm{B}_r(\vec{0})} \left(\tilde{\beta}_F(x)\right)^pdx\right)^{1/p}\leq \theta \quad \text{for} \quad \tilde{\beta}_F(x) \defeq \sup_{\mathrm{X} \in \text{Sym}(n) \atop{\mathrm{X} \neq \mathbf{O}_n}} \frac{|F(\mathrm{X}, x)-F(\mathrm{X}, \vec{0})|}{\|\mathrm{X}\|}
$$
 holds for all $r \leq \min\{r_0(\theta), \operatorname{dist}(\Omega', \partial\Omega)\}$, for some $r_0 = r_0(\theta) > 0$, then $u \in C^{1,\alpha}_{\textnormal{loc}}(\Omega)$, and
\begin{equation} \label{eq:1.9}
    \|u\|_{C^{1,\alpha}(\overline{\Omega'})} \leq \mathrm{C}\, (\left\| u\right\|_{L^{\infty }(\Omega )} + \|f\|_{L^p(\Omega)}), 
\end{equation}
for any subdomain $\Omega' \subset\subset \Omega$, where $\mathrm{C}>0$ depends only on $r_0, n, p, \lambda, \Lambda, \alpha, \operatorname{diam}(\Omega), \|\mathfrak{B}\|_{L^{\infty}(\mathrm{B}_1)}$, and $\operatorname{dist}(\Omega', \partial\Omega)$.
\end{theorem}

\medskip

Finally, recall that viscosity solutions satisfy the following stability property (see, for instance, \cite[Proposition 4.11]{CafCabre1995}) and  \cite[Theorem 3.8]{CCKS}.

\begin{proposition}\label{Stability-Prop}
Let $(F_k)_{k \in \mathbb{N}}$ be a sequence of fully nonlinear elliptic operators with ellipticity constants $0< \lambda \leq \Lambda$. Let $(u_k)_{k \in \mathbb{N}} \subset C^0(\Omega)$ be viscosity solutions to
\[
F_k(D^2 u_k, x) +\langle \mathfrak{B}_k(x),Du_k\rangle = f_k(x) \quad \text{in } \Omega,
\]
where $(f_k)_{k \in \mathbb{N}}$ is a sequence of continuous functions and $(\mathfrak{B}_k)_{k \in \mathbb{N}}$ is a sequence of continuous vector fields. Suppose further that $F_k \to F$ locally uniformly in $\text{Sym}(n)$, and that $u_k\to u$ locally uniformly in $\Omega$, $\mathfrak{B}_k\to \mathfrak{B}$ and $f_k \to f$ a.e. in $\Omega$. Then,
\[
F(D^2 u, x) +\langle \mathfrak{B}(x),Du\rangle = f(x) \quad \text{in } \Omega
\]
in the viscosity sense.
\end{proposition}

\section{Geometric Tangential Approach}

In what follows, under suitable smallness assumptions on the data, we establish a key result that provides a tangential approach to the regularity theory available for constant-coefficient, homogeneous \( F \)-harmonic functions (which enjoy higher regularity estimates).

\begin{lemma}[{\bf \( F \)-harmonic approximation lemma}]\label{aproxlemma}
Let \( \hat{\tau}:[0, \infty) \to [0, \infty)\) be a modulus of continuity satisfying the limiting compatibility condition
\begin{equation}\label{LCC}\tag{LCC}
\lim_{t \to 0^{+}} \frac{\hat{\tau}(t)}{t} = +\infty,
\end{equation}
and let \( u \in \mathscr{F} \). Then, there exists a positive constant \( \eta = \eta(n,p_{0}, \lambda, \Lambda, \omega, \hat{\tau}) \) such that, for any \( \mu > 0 \), if \( u \) is a normalized viscosity solution (i.e., $\|u\|_{L^{\infty}(\mathrm{B}_1)} \leq 1$) of
\[
\frac{1}{\mu} F(\mu D^{2} u, x) + \langle \mathfrak{B}(x), D u \rangle = f(x) \quad \text{in} \quad \mathrm{B}_{1},
\]
and
\[
\max \left\{ \mu, \| \theta_{F} \|_{L^{p_{0}}(\mathrm{B}_{1})}, \| \mathfrak{B} \|_{L^{\infty}(\mathrm{B}_{1}; \mathbb{R}^{n})}, |u(0)|, |D u(0)|, \| f \|_{L^{p_{0}}(\mathrm{B}_{1})} \right\} \leq \eta,
\]
then there exists a number \( \rho \in (0, 1/2) \), depending only on \( n \), \( \lambda \), \( \Lambda \), and \( \hat{\tau} \), and a quadratic polynomial \( \mathfrak{P} \) of the form \( \mathfrak{P}(x) = \frac{1}{2} x^{T} \mathrm{M} x \), with
\[
\frac{1}{\mu} F( \mu D^{2} \mathfrak{P}, 0 ) = 0 \quad \text{in} \quad \mathrm{B}_{1/2},
\]
and \( \| \mathrm{M} \| \leq \mathrm{C}(n, \lambda, \Lambda) \), such that
\[
\sup_{\mathrm{B}_{\rho}} | u - \mathfrak{P} | \leq \rho^{2} \hat{\tau}(\rho).
\]
\end{lemma}

\begin{proof}
We prove the result by means of a \textit{Reductio ad Absurdum} argument. In this case, for some \(0 < \rho_{0} < 1\) to be chosen \textit{a posteriori}, we can find sequences \((u_{j})_{j \in \mathbb{N}}\), \((\mathfrak{B}_{j})_{j \in \mathbb{N}}\), \((f_{j})_{j \in \mathbb{N}}\), \((\mu_{j})_{j \in \mathbb{N}} \subset (0,\infty)\), and \((F_{j})_{j \in \mathbb{N}} \subset \mathscr{F}\) such that \(u_{j}\) is a normalized viscosity solution of
\begin{equation}\label{eq1lemmaaprox}
\frac{1}{\mu_{j}} F_{j}(\mu_{j} D^{2} u_{j}, x) + \langle \mathfrak{B}_{j}(x), D u_{j} \rangle = f_{j}(x) \quad \text{in} \quad \mathrm{B}_{1},
\end{equation}
where
\begin{equation}\label{eq2lemmaaprox}
\max \left\{ \mu_{j}, \| \theta_{F_{j}} \|_{L^{p_{0}}(\mathrm{B}_{1})}, \| \mathfrak{B}_{j} \|_{L^{\infty}(\mathrm{B}_{1}; \mathbb{R}^{n})}, |u_{j}(0)|, |D u_{j}(0)|, \| f_{j} \|_{L^{p_{0}}(\mathrm{B}_{1})} \right\} \leq \frac{1}{j},
\end{equation}
however, for every quadratic polynomial \(\mathfrak{P}\) of the form \(\mathfrak{P}(x) = \frac{1}{2} x^{T} \mathfrak{M} x\) satisfying
\begin{equation*}
\frac{1}{\mu_{j}} F_{j}(\mu_{j} D^{2} \mathfrak{P}, 0) = 0 \quad \text{in} \quad \mathrm{B}_{1/2},
\end{equation*}
it holds that
\begin{equation}\label{eq3lemmaaprox}
\sup_{\mathrm{B}_{\rho_{0}}} |u_{j} - \mathfrak{P}| > \rho_{0}^{2} \hat{\tau}(\rho_{0}).
\end{equation}

Since \(F_{j} \in \mathscr{F}\), by passing to a subsequence if necessary, we may assume that \(F_{j}(\mathrm{X}, x) \to F_{\infty}(\mathrm{X}, x)\) locally uniformly in \(\mathrm{Sym}(n)\) for all \(x \in \mathrm{B}_{1}\), for some operator \(F_{\infty} \in C^{1}\) that is uniformly elliptic. It follows that
\begin{equation}\label{eq4lemmaaprox}
\frac{1}{\mu_{j}} F_{j}(\mu_{j} \mathrm{X}, x) \to D_{\mathrm{M}} F_{\infty}(\mathbf{O}_n) \mathrm{X} = \operatorname{Tr}(\mathfrak{A} \mathrm{X}),
\end{equation}
locally uniformly in \(\mathrm{Sym}(n)\); see, e.g., \cite{João,dosPrazTei2016} for further details. Moreover, from \eqref{eq2lemmaaprox}, we obtain that \(\mathfrak{B}_{j} \to 0\), \(\mu_{j} \to 0\), and \(f_{j} \to 0\) in \(L^{p_{0}}\)-norm. Also, by the \(C^{1, \alpha}\)-regularity (Theorem \ref{C1alpharegularity}) for solutions to \eqref{eq1lemmaaprox}, up to a subsequence, we have \(u_{j} \to u_{\infty} \in C^{1, \alpha}_{\mathrm{loc}}(\mathrm{B}_{1})\) locally uniformly in \(\mathrm{B}_{1}\), with \(u_{\infty}(0) = |D u_{\infty}(0)| = 0\).

By the stability of viscosity solutions (Proposition \ref{Stability-Prop}), we conclude that \(u_{\infty}\) solves
\begin{equation}\label{eq5lemmaaprox}
\mathrm{Tr}(\mathfrak{A}_0D^{2} u_{\infty}) = D_{\mathrm{M}} F_{\infty}(\mathbf{O}_n, 0) D^{2} u_{\infty} = 0 \quad \text{in} \quad \mathrm{B}_{3/4}
\end{equation}
in the viscosity sense. Since this equation has constant coefficients and linear elliptic structure, \(u_{\infty}\) is smooth. We define
\[
\overline{\mathfrak{P}}(x) = \frac{1}{2} x^{T} D^{2} u_{\infty}(0) x.
\]
As \(u_{\infty}\) is normalized (since \(\|u_{j}\|_{L^{\infty}(\mathrm{B}_{1})} \leq 1\) for all \(j \in \mathbb{N}\)), it follows from standard estimates for \(u_{\infty}\) (see \cite[Theorem 2.5.2]{Kry96} and \cite[Lemma 3.1]{João}) that
\begin{equation}\label{eq6lemmaaprox}
\sup_{\mathrm{B}_{r}} |u_{\infty} - \overline{\mathfrak{P}}| \leq \mathrm{C}_{\mathrm{H}} r^{3}, \quad \forall r \in \left(0, \frac{1}{2}\right),
\end{equation}
for some universal constant \(\mathrm{C}_{\mathrm{H}} > 0\). This, combined with the assumption on the modulus of continuity \(\hat{\tau}\), i.e., \eqref{LCC}, ensures the existence of a constant \(0 < \rho_{0} \ll \frac{1}{2}\) such that
\[
\rho_{0} \leq \frac{1}{3 \mathrm{C}_{\mathrm{H}}} \hat{\tau}(\rho_{0}).
\]
Thus, from \eqref{eq6lemmaaprox}, we obtain
\begin{equation}\label{eq7lemmaaprox}
\sup_{\mathrm{B}_{\rho_{0}}} |u_{\infty} - \overline{\mathfrak{P}}| \leq \frac{1}{3} \rho_{0}^{2} \hat{\tau}(\rho_{0}).
\end{equation}

Now, from \eqref{eq5lemmaaprox}, we conclude that
\[
D_{\mathrm{M}} F_{\infty}(\mathbf{O}_n, 0) D^{2} \overline{\mathfrak{P}} = 0 \quad \text{in} \quad \mathrm{B}_{1/2},
\]
and thus,
\[
\left| \frac{1}{\mu_{j}} F_{j}(\mu_{j} D^{2} \overline{\mathfrak{P}}, 0) \right| = o(1) \quad \text{as} \quad j \to \infty.
\]
By applying the arguments from \cite[Lemma 3.1]{dosPrazTei2016} and \cite[Lemma 1]{João}, the uniform ellipticity of \(F_{j}\) and the normalization \(F_{j}(\mathbf{O}_n, 0) = 0\) imply the existence of a sequence \((a_{j})_{j \in \mathbb{N}}\) with \(|a_{j}| = o(1)\) as \(j \to \infty\) such that, for each \(j\), the quadratic polynomial
\[
\mathfrak{P}_{j}(x) = \overline{\mathfrak{P}}(x) + \frac{1}{2} a_{j} |x|^{2}
\]
satisfies
\[
\frac{1}{\mu_{j}} F_{j}(\mu_{j} D^{2} \mathfrak{P}_{j}, 0) = 0 \quad \text{in} \quad \mathrm{B}_{1/2}.
\]
Thus, using \eqref{eq7lemmaaprox}, the definition of \(\mathfrak{P}_{j}\), and the uniform convergence \(u_{j} \to u_{\infty}\), we deduce
\begin{align*}
\sup_{\mathrm{B}_{\rho_{0}}} |u_{j} - \mathfrak{P}_{j}| & \leq \sup_{\mathrm{B}_{\rho_{0}}} |u_{j} - u_{\infty}| + \sup_{\mathrm{B}_{\rho_{0}}} |u_{\infty} - \overline{\mathfrak{P}}| + \sup_{\mathrm{B}_{\rho_{0}}} |\mathfrak{P}_{j} - \overline{\mathfrak{P}}| \\
& \leq \frac{2}{3} \rho_{0}^{2} \hat{\tau}(\rho_{0}) + \frac{1}{2} |a_{j}| \rho_{0}^{2} \\
& \leq \rho_{0}^{2} \hat{\tau}(\rho_{0})
\end{align*}
for \(j \gg 1\), which contradicts \eqref{eq3lemmaaprox} and concludes the proof.
\end{proof}

\begin{remark}
In Lemma~\ref{aproxlemma}, it is sufficient to assume that \(F(\mathbf{O}_{n},x)=0\) at \(x=0\), due to the pointwise nature of the estimates.
Accordingly, whenever this result is invoked, it suffices that the operator
\(F:\mathrm{Sym}(n)\times \mathrm{B}_{1}\to\mathbb{R}\)
satisfies this condition after normalization; see Proposition~\ref{Prop5.1} for details on the normalization procedure.
\end{remark}

Next, we apply the approximation tool for normalized solutions under a smallness condition on the \(L^{\infty}\)-norm of the solution. In summary, we obtain the following result.

\begin{lemma}[{\bf First step of induction}]\label{approxlemmadelta}
Let \(\hat{\tau}: [0, \infty) \to [0, \infty)\) be a modulus of continuity satisfying \eqref{LCC}
and let \(F \in \mathscr{F}\). Then, there exists a constant \(0 < \delta \ll 1\), depending only on \(n\), \(p_{0}\), \(\lambda\), \(\Lambda\), \(\omega\), and \(\hat{\tau}\), such that if \(u\) is a viscosity solution of
\[
F(D^{2} u, x) + \langle \mathfrak{B}(x), D u \rangle = f(x) \quad \text{in} \quad \mathrm{B}_{1},
\]
and
\[
\|u\|_{L^{\infty}(\mathrm{B}_{1})} \leq \delta \quad \text{and} \quad \max \left\{ \| \theta_{F} \|_{L^{p_{0}}(\mathrm{B}_{1})}, \| \mathfrak{B} \|_{L^{\infty}(\mathrm{B}_{1}; \mathbb{R}^{n})}, |u(0)|, |D u(0)|, \| f \|_{L^{p_{0}}(\mathrm{B}_{1})} \right\} \leq \sqrt{\delta^{3}},
\]
then there exist a radius \(\rho \in (0, 1/2)\), depending only on \(n\), \(p_{0}\), \(\lambda\), \(\Lambda\), and \(\hat{\tau}\), and a quadratic polynomial \(\mathfrak{P} : \mathrm{B}_{1/2} \to \mathbb{R}\) of the form \(\mathfrak{P}(x) = \frac{1}{2} x^{T} \mathrm{M} x\), satisfying
\[
F(D^{2} \mathfrak{P}, 0) = 0 \quad \text{in} \quad \mathrm{B}_{1/2}
\]
and \(\| \mathrm{M} \| \leq \mathrm{C}_{\mathscr{H}}(n,p_{0} \lambda, \Lambda) \delta\), such that
\[
\sup_{\mathrm{B}_{\rho}} |u - \mathfrak{P}| \leq \delta \rho^{2} \hat{\tau}(\rho).
\]
\end{lemma}

\begin{proof}
Fix \(\delta > 0\) to be chosen \textit{a posteriori}, and define the normalized function \(w(x) = \delta^{-1} u(x)\). It is straightforward to verify that \(w\) satisfies
\[
\frac{1}{\delta} F(\delta D^{2} w, x) + \langle \mathfrak{B}(x), D w \rangle = f_{\delta}(x) \quad \text{in} \quad \mathrm{B}_{1},
\]
in the viscosity sense, where \(f_{\delta}(x) = \delta^{-1} f(x)\). Given the constant \(\eta\) from Lemma \ref{aproxlemma}, the desired result holds by choosing
\[
\delta = \min \left\{ \eta^{2}, \eta^{1/3} \right\}.
\]
\end{proof}

\section{$C^{2, \text{Dini}}$ estimates: Proof of Theorem \ref{MainThm01}}

In this section, we present the proof of Theorem \ref{MainThm01}. To this end, we aim to establish the desired conclusion by iteratively applying the result of Lemma \ref{approxlemmadelta} through an induction argument to obtain pointwise estimates. Subsequently, we employ a standard covering argument to complete the proof.

\begin{proposition}\label{Prop5.1}
Assume that conditions \((\mathrm{A1})-(\mathrm{A4})\) hold. Then, there exists a constant \(\delta>0\), depending only on \(n\), \(p_{0}\), \(\lambda\), \(\Lambda\), \(\omega\), and \(\tau\), such that if
\[
\|u\|_{L^{\infty}(\mathrm{B}_{1})} \leq \delta,
\]
then \(u \in C^{2,\psi}\) at the origin. That is, there exists a quadratic polynomial \(\mathfrak{P}\) satisfying
\[
\|u - \mathfrak{P}\|_{L^{\infty}(\mathrm{B}_{r})} \leq \mathrm{C}_0 \delta\, r^{2} \psi(r), \quad \text{for } r \in (0, 1/2),
\]
where \(\mathrm{C}_{0} = \mathrm{C}_{0}(n, p_{0},\lambda, \Lambda, \tau, \alpha_{0})\) is a positive constant and  $\displaystyle \psi(r) = \tau(r)+\int_{0}^{r} \frac{\tau(s)}{s}ds$.
\end{proposition}

\begin{proof}
Initially, we know that \(u \in C^{1,\alpha}_{\mathrm{loc}}\) by Theorem \ref{C1alpharegularity}. In particular, \(Du\) is well-defined in \(\mathrm{B}_{1/2}\), and we have
\begin{equation}\label{est1Lemma5.1}
\|u\|_{C^{1,\alpha}(\mathrm{B}_{1/2})} \leq \mathrm{C}_{0} \left( \|u\|_{L^{\infty}(\mathrm{B}_{1})} + \|f\|_{L^{p_{0}}(\mathrm{B}_{1})} \right) \defeq \mathrm{C}_{1}.
\end{equation}

With this observation in mind, we may assume, without loss of generality, that \(u(0) = f(0) = |Du(0)| = 0\) and that the assumptions of Lemma \ref{approxlemmadelta} are satisfied for \(u\), \(\mathfrak{B}\), \(f\), and \(\theta_{F}\) (cf. \cite{Lian-Wang-Zhang2024}). 
Otherwise, since the map \(F(\cdot,0): \mathrm{Sym}(n) \to \mathbb{R}\) is surjective by the structural condition {\rm(A1)},
there exists a matrix \(\mathfrak{A}_{0} \in \mathrm{Sym}(n)\) such that
\[
F(\mathfrak{A}_{0},0)= f(0)-\langle \mathfrak{B}(0), Du(0)\rangle .
\]
Accordingly, it suffices to consider the rescaled function
\[
\tilde{u}(x)
\coloneqq
\frac{u(rx)-u(0)-\langle Du(0), rx\rangle-\tfrac{r^{2}}{2}x^{T}\mathfrak{A}_{0}x}{r^{2}\mathcal{K}},
\]
which is a viscosity solution of
\[
\tilde{F}(D^{2}\tilde{u},x)+\langle \tilde{\mathfrak{B}}(x),D\tilde{u}\rangle=\tilde{f}(x)
\quad \text{in } \mathrm{B}_{1},
\]
where
\[
\tilde{F}(\mathrm{M},x)
\coloneqq
\mathcal{K}^{-1}\!\left[ F(\mathcal{K}\mathrm{M}+\mathfrak{A}_{0},rx)-F(\mathfrak{A}_{0},0)\right],
\]
and
\[
\tilde{f}(x)
\coloneqq
\mathcal{K}^{-1}\!\left[ f(rx)-f(0)
-\langle \mathfrak{B}(rx)-\mathfrak{B}(0),Du(0)\rangle
+\langle \mathfrak{B}(rx),\mathfrak{A}_{0}rx\rangle \right],
\qquad
\tilde{\mathfrak{B}}(x)\coloneqq r\,\mathfrak{B}(rx).
\]
We now set
\[
r \coloneqq
\min\!\left\{
r_{0},
\tau^{-1}\!\left(\frac{\delta^{3/2}}{|\mathrm{B}_{1}|^{1/p_{0}}\mathrm{C}_{\theta_{F}}}\right),
\frac{\delta^{3/2}}{1+\|\mathfrak{B}\|_{L^{\infty}(\mathrm{B}_{1};\mathbb{R}^{n})}}
\right\} \quad \text{and} \quad 
\mathcal{K} \coloneqq
\max\!\left\{
1,
\frac{\mathrm{C}_{1}r^{\alpha-1}+\|\mathfrak{A}_{0}\|}{\delta},
\frac{\mathrm{C}_{2}}{\delta^{3/2}},
\frac{\mathrm{C}_{2}}{\mathrm{C}_{f}}
\right\},
\]
where \(\mathrm{C}_{1}>0\) is the constant appearing in \eqref{est1Lemma5.1},
\[
\mathrm{C}_{2}
\coloneqq
|\mathrm{B}_{1}|^{1/p_{0}}
\Big(
\mathrm{C}_{f}
+(2\|\mathfrak{A}_{0}\|+\mathrm{C}_{1})\mathrm{C}_{\mathfrak{B}}
+\|\mathfrak{B}\|_{L^{\infty}(\mathrm{B}_{1})}\|\mathfrak{A}_{0}\|
\Big)
\max\{1,\tau(1)\},
\]
and \(r_{0}\ll1\) is chosen so that, for all \(0<s<r_{0}\),
\[
s\leq \tau(s),
\]
cf.\ Remark~\ref{Remark1.4}.

With these choices, it is straightforward to verify that \(\tilde{u}\) satisfies the assumptions of Lemma~\ref{approxlemmadelta}, and that
\[
\tilde{u}(0)=|D\tilde{u}(0)|=\tilde{f}(0)=\tilde{F}(\mathbf{O}_{n},0)=0.
\]

To prove the desired result, we claim that there exist constants \(\rho_{0} \in (0, e^{-1}]\), \(\mathrm{C}_{\mathscr{H}} > 0\), and a sequence of quadratic polynomials \((\mathfrak{P}_{k})_{k \geq 0}\) of the form
\[
\mathfrak{P}_{k}(x) = \frac{1}{2} x^{T} \mathrm{M}_{k} x
\]
such that:
\begin{itemize}
\item[(i)] \(F(\mathrm{M}_{k}, 0) = 0\);
\item[(ii)] \(\displaystyle \sup_{\mathrm{B}_{\rho_{0}^{k}}} |u - \mathfrak{P}_{k}| \leq \delta \rho_{0}^{2k} \tau(\rho_{0}^{k})\);
\item[(iii)] \(\| \mathrm{M}_{k} - \mathrm{M}_{k-1} \| \leq \mathrm{C}_{\mathscr{H}} \delta \tau(\rho_{0}^{k-1})\);
\end{itemize}
with \(\mathfrak{P}_{0} = \mathfrak{P}_{-1} = 0\). In fact, we set \(\mathrm{C}_{\mathscr{H}} > 0\) as the constant given by Lemma \ref{approxlemmadelta}. Now, given \(\rho \in (0, e^{-1}]\), by the Dini continuity of \(\tau\) and the definition of \(\psi\), we apply the integral test to obtain
\begin{equation}\label{est2Lemma5.1}
\sum_{k=1}^{\infty} \tau(\rho^{k-1}) \leq \tau(1) + \int_{1}^{\infty} \tau(\rho^{t-1}) \, dt = \tau(1) + \frac{1}{\ln \left(\frac{1}{\rho}\right)} \int_{0}^{1} \frac{\tau(s)}{s} \, ds \leq \psi(1).
\end{equation}

Next, assumption \((\mathrm{A4})\) implies the existence of a constant \(\rho_{1} \in (0, r_{0})\) such that
\begin{equation}\label{est3Lemma5.1}
\sup_{0 \leq r \leq 1} \frac{\tau(r \rho_{1})}{\tau(r)} \leq \frac{\delta^{3/2}}{(1 + \mathrm{C}_{\mathscr{H}} \psi(1)) \mathrm{C}_{\theta_{F}}}, \quad \sup_{k \in \mathbb{N}} \frac{\rho_{1}^{\alpha_{0}} \tau(\rho_{1}^{k})}{\tau(\rho_{1}^{k+1})} \leq 1,
\end{equation}
and
\begin{equation}\label{est3'Lemma5.1}
\sup_{0 < r \leq \rho_{1}} \frac{r}{\tau(r)} \leq \frac{\delta^{1/2}}{2 \left( 1 + |\mathrm{B}_{1}|^{1/p_{0}} \psi(1) \mathrm{C}_{\mathscr{H}} \| \mathfrak{B} \|_{L^{\infty}(\mathrm{B}_{1}; \mathbb{R}^{n})} \right)}.
\end{equation}
We may also assume, without loss of generality, that \(\mathrm{C}_{f} \leq \frac{\delta^{3/2}}{2 |\mathrm{B}_{1}|^{1/p_{0}}}\) and \(\mathrm{C}_{\theta_{F}} \leq \frac{\delta^{3/2}}{|\mathrm{B}_{1}|^{1/p_{0}} (1 + \mathrm{C}_{\mathscr{H}} \delta \psi(1))}\) (For clarity, we invoke here a standard normalization and scaling argument, analogous to the one employed in the derivation of the Schauder estimates in Caffarelli--Cabré's monograph \cite[Ch.~8, Thm.~8.1]{CafCabre1995}
).

With these reductions, we proceed by induction. Indeed, we may select \(0 < \rho_{0} \leq \min \{ e^{-1}, \rho, \rho_{1} \}\). The case \(k = 0\) is immediate, and the case \(k = 1\) follows from Lemma \ref{approxlemmadelta} by taking \(\hat{\tau} = \tau\).

Now, assume that the statement holds for some \(k\geq 1\). We define the auxiliary function
\[
v_{k}(x) = \frac{(u - \mathfrak{P}_{k})(\rho_{0}^{k} x)}{\rho_{0}^{2k} \tau(\rho_{0}^{k})},
\]
which solves in the viscosity sense
\[
F_{k}(D^{2} v_{k}, x) + \langle \mathfrak{B}_{k}(x), D v_{k} \rangle = f_{k}(x) \quad \text{in } \mathrm{B}_{1},
\]
where
\[
\left\{
\begin{array}{rcl}
F_{k}(\mathrm{M}, x) & \defeq & \frac{1}{\tau(\rho_{0}^{k})} F \left( \tau(\rho_{0}^{k}) \mathrm{M} + \mathrm{M}_{k}, \rho_{0}^{k} x \right), \\
\mathfrak{B}_{k}(x) & \defeq & \rho_{0}^{k} \mathfrak{B}(\rho_{0}^{k} x), \\
f_{k}(x) & \defeq & \frac{1}{\tau(\rho_{0}^{k})} \left[ f(\rho_{0}^{k} x) - \langle \mathfrak{B}(\rho_{0}^{k} x), \rho_{0}^{k} \mathrm{M}_{k} x \rangle \right].
\end{array}
\right.
\]
Since \(\rho_{0} \leq 1\), it is clear that \(\| \mathfrak{B}_{k} \|_{L^{\infty}(\mathrm{B}_{1}; \mathbb{R}^{n})} \leq \delta^{3/2}\). Moreover, since \(\mathrm{M}_{0} = 0\), it follows from the induction hypothesis and \eqref{est2Lemma5.1} that
\begin{equation}\label{est4Lemma5.1}
\| \mathrm{M}_{k} \| \leq \sum_{j = 1}^{k} \| \mathrm{M}_{j} - \mathrm{M}_{j-1} \| \leq \mathrm{C}_{\mathscr{H}} \delta \sum_{j = 1}^{\infty} \tau(\rho_{0}^{j-1}) \leq \mathrm{C}_{\mathscr{H}} \delta \psi(1).
\end{equation}

On the other hand, it can be verified that
\begin{eqnarray}
\|f_{k}\|_{L^{p_{0}}(\mathrm{B}_{1})} &\leq& \frac{|\mathrm{B}_{1}|^{1/p_{0}}}{\tau(\rho_{0}^{k})} \left[ \left( \intav{\mathrm{B}_{\rho_{0}^{k}}} |f(x)|^{p_{0}} \, dx \right)^{1/p_{0}} + \|\mathfrak{B}\|_{L^{\infty}(\mathrm{B}_{1}; \mathbb{R}^{n})} \|\mathrm{M}_{k}\| \rho_{0}^{k} \right] \nonumber \\
&\leq& |\mathrm{B}_{1}|^{1/p_{0}} \mathrm{C}_{f} + \frac{\rho_{0}^{k}}{\tau(\rho_{0}^{k})} |\mathrm{B}_{1}|^{1/p_{0}} \|\mathfrak{B}\|_{L^{\infty}(\mathrm{B}_{1}; \mathbb{R}^{n})} \mathrm{C}_{\mathscr{H}} \delta \psi(1) \nonumber \\
&\leq& \delta^{3/2}, \label{est5Lemma5.1}
\end{eqnarray}
where we have used estimates \eqref{est3'Lemma5.1} and \eqref{est4Lemma5.1}. Moreover, by the bound on \(\|\mathrm{M}_{k}\|\), we obtain
\[
\theta_{F_{k}}(x) \leq \frac{1 + \mathrm{C}_{\mathscr{H}} \delta \psi(1)}{\tau(\rho_{0}^{k})} \theta_{F}(\rho_{0}^{k} x).
\]
Consequently, \(\| \theta_{F_{k}} \|_{L^{p_{0}}(\mathrm{B}_{1})} \leq \delta^{3/2}\). Since \(v_{k}(0) = |D v_{k}(0)| = 0\) and \(F_{k}\) satisfies \(\mathrm{(A1)}\)–\(\mathrm{(A2)}\) with the same ellipticity constants and modulus of continuity, we can apply Lemma \ref{approxlemmadelta} once again with \(\hat{\tau}(t) = t^{\alpha_{0}}\) to conclude that there exists a quadratic polynomial \(\overline{\mathfrak{P}}(x) = \frac{1}{2} x^{T} \overline{\mathrm{M}} x\) with universally bounded coefficients such that
\begin{equation}
\sup_{\mathrm{B}_{\rho_{0}}} |v_{k} - \overline{\mathfrak{P}}| \leq \delta \rho_{0}^{2 + \alpha_{0}} \quad \text{and} \quad F_{k}(\overline{\mathrm{M}}, 0) = 0.
\end{equation}

We now define the quadratic polynomial
\[
\mathfrak{P}_{k+1}(x) = \mathfrak{P}_{k}(x) + \rho_{0}^{2k} \tau(\rho_{0}^{k}) \overline{\mathfrak{P}}(\rho_{0}^{-k} x)
\]
and observe that
\[
F(\mathrm{M}_{k+1}, 0) = \tau(\rho_{0}^{k}) F_{k}(\overline{\mathrm{M}}, 0) = 0, \quad \|\mathrm{M}_{k+1} - \mathrm{M}_{k}\| \leq \tau(\rho_{0}^{k}) \|\overline{\mathrm{M}}\| \leq \mathrm{C}_{\mathscr{H}} \delta \tau(\rho_{0}^{k}),
\]
and
\[
\sup_{\mathrm{B}_{\rho_{0}^{k+1}}} |u - \mathfrak{P}_{k+1}| = \rho_{0}^{2k} \tau(\rho_{0}^{k}) \sup_{\mathrm{B}_{\rho_{0}}} |v_{k} - \overline{\mathfrak{P}}| \leq \rho_{0}^{2(k+1)} \delta \tau(\rho_{0}^{k}) \rho_{0}^{\alpha_{0}}.
\]
By estimate \eqref{est3Lemma5.1}, it follows that
\[
\sup_{\mathrm{B}_{\rho_{0}^{k+1}}} |u - \mathfrak{P}_{k+1}| \leq \delta \rho_{0}^{2(k+1)} \tau(\rho_{0}^{k+1}).
\]
This completes the proof of the claim.

By the construction of the sequence of quadratic polynomials \((\mathfrak{P}_{k})\), we deduce that \((\mathrm{M}_{k})_{k \geq 0}\) forms a Cauchy sequence in \(\mathrm{Sym}(n)\), and hence there exists \(\displaystyle \mathrm{M}_{\infty} = \lim_{k \to \infty} \mathrm{M}_{k}\). Furthermore, by condition \((iii)\), we have
\begin{equation}\label{est7Lemma5.1}
\|\mathrm{M}_{k} - \mathrm{M}_{\infty}\| \leq \sum_{j = k + 1}^{\infty} \| \mathrm{M}_{j} - \mathrm{M}_{j - 1} \| \leq \mathrm{C}_{\mathscr{H}} \delta \sum_{j = k + 1}^{\infty} \tau(\rho_{0}^{j - 1}).
\end{equation}

Finally, define \(\mathfrak{P}_{\infty}(x) = \frac{1}{2} x^{T} \mathrm{M}_{\infty} x\) and fix \(r \in (0, \rho_{0}]\). In this case, we can select \(k \in \mathbb{N}\) such that \(\rho_{0}^{k+1} < r \leq \rho_{0}^{k}\). Then, by condition \((ii)\), estimate \eqref{est7Lemma5.1}, and the integral test, we obtain
\begin{eqnarray*}
\sup_{\mathrm{B}_{r}} |u - \mathfrak{P}_{\infty}| &\leq& \sup_{\mathrm{B}_{\rho_{0}^{k}}} |u - \mathfrak{P}_{k}| + \sup_{\mathrm{B}_{\rho_{0}^{k}}} |\mathfrak{P}_{k}(x) - \mathfrak{P}_{\infty}(x)| \\
&\leq& \delta \rho_{0}^{2k} \tau(\rho_{0}^{k}) + \mathrm{C}_{\mathscr{H}} \delta \rho_{0}^{2k} \sum_{j = k + 1}^{\infty} \tau(\rho_{0}^{j - 1}) \\
&\leq& \frac{\delta}{\rho_{0}^{2 + \alpha_{0}}} \rho_{0}^{2(k + 1)} \left( \tau(\rho_{0}^{k + 1}) + \mathrm{C}_{\mathscr{H}} \sum_{j = k + 1}^{\infty} \tau(\rho_{0}^{j}) \right) \\
&\leq& \frac{\delta}{\rho_{0}^{2 + \alpha_{0}}} r^{2} (1 + \mathrm{C}_{\mathscr{H}}) \psi(\rho_{0}^{k + 1}) \\
&\leq& \mathrm{C}_{0} \delta r^{2} \psi(r),
\end{eqnarray*}
where \(\mathrm{C}_{0} = \rho_{0}^{-(2 + \alpha_{0})} (1 + \mathrm{C}_{\mathscr{H}})\). This concludes the proof.
\end{proof}

\medskip

In conclusion, we will address the proof of the Evans-Krylov-type result: 

\begin{proof}[{\bf Proof of Corollary \ref{Corolario-EK}}]
As previously established, it suffices to show that $u$ belongs to the class $C^{2,\psi}$ at the origin. To this end, for some $r > 0$ to be chosen \textit{a posteriori}, we define $v : \mathrm{B}_1 \to \mathbb{R}$ by
\[
v(x) \defeq \frac{1}{r^2} u(rx) - \left[ \frac{1}{r^2} u(0) + \frac{1}{r} D u(0) \cdot x + \frac{1}{2} x^T D^2 u(0) x \right].
\]
It is straightforward to verify that
\begin{equation} \label{eq:normalization}
    v(0) = 0, \quad |D v(0)| = 0, \quad \text{and} \quad |D^2 v(x)| = |D^2 u(rx)-D^2 u(0)| \leq \varsigma(r),
\end{equation}
where $\varsigma: [0, \infty) \to [0, \infty)$ denotes the modulus of continuity of $D^2 u$.

We now select $0 < r \ll 1$ sufficiently small so that $\varsigma(r) \leq \mathrm{c}_n \delta_0$ and $\frac{r}{\tau(r)}< 1$ (Recall Remark~\ref{Remark1.4}), where $\mathrm{c}_n > 0$ is a dimensional constant and $\delta_0$ is the constant appearing in Theorem \ref{MainThm01}. With this choice, the function $v$ satisfies
\[
 \widetilde{F}(D^{2}v,x) +  \langle  \widetilde{\mathfrak{B}}(x), Dv\rangle =  \widetilde{f}(x) \quad \text{in} \quad \mathrm{B}_1,
\]
where
\[
\left\{
\begin{array}{rcl}
  \widetilde{F}(\mathrm{M}, x)   & \defeq &  F(\mathrm{M} + D^2 u(0), rx),\\
  \widetilde{\mathfrak{B}}(x)   & \defeq  & r\, \mathfrak{B}(rx), \\
  \widetilde{f}(x)   & \defeq & f(rx) - \left\langle r\, \mathfrak{B}(rx), \frac{1}{r} D u(0) + D^{2}u(0)x \right\rangle.
\end{array}
\right.
\]
Moreover, the structural assumptions \(\rm{(A1)}-\rm{(A4)}\) remain satisfied. Additionally, we emphasize that
\[
\tilde{\tau}(t) = \tilde{\mathrm{C}}\cdot \tau(rt),
\]
where
{\scriptsize{
$$
\tilde{\mathrm{C}} = \max\left\{\mathrm{C}_{\mathfrak{B}}, (1+\|D^2u\|_{L^{\infty}(\mathrm{B}_1)})\mathrm{C}_{\theta_F}, \mathrm{C}_{f} + \mathrm{C}_{\mathfrak{B}}\|Du\|_{L^{\infty}(\mathrm{B}_{1})}+  \mathrm{C}_{\mathfrak{B}} \|D^2 u\|_{L^{\infty}(\mathrm{B}_1)} +  \|\mathfrak{B}\|_{L^{\infty}(\mathrm{B}_1; \mathbb{R}^n)}\|D^2 u\|_{L^{\infty}(\mathrm{B}_1)\frac{rt}{\tau(rt)}}\right\}.
$$}}
Therefore, by applying Theorem \ref{MainThm01}, we obtain the desired regularity result. After performing a covering argument, we conclude that  \(v \in C^{2,\psi}(\mathrm{B}_{1/2})\) and consequently \(u \in C^{2,\psi}(\mathrm{B}_{r/2})\).
\end{proof}

\subsection*{Acknowledgments}

J. da Silva Bessa has been supported by FAPESP-Brazil under Grant No. 2023/18447-3. J.V. da Silva has received partial support from CNPq-Brazil under Grant No. 307131/2022-0, Chamada CNPq/MCTI No. 10/2023 - Faixa B - Consolidated Research Groups under Grant No. 420014/2023-3, and FAPESP-Brazil under the Grant No.  2025/09344-1 - Special Programs - Special Projects - First Projects - Call for Proposals (2025) - 1st Cycle. J.V. da Silva and L. Ospina have been supported by FAEPEX-UNICAMP (Project No. 2441/23, Special Calls - PIND - Individual Projects, 03/2023). We would like to express our sincere gratitude to the anonymous referees for their insightful comments and constructive suggestions, which substantially enhanced the final version of this manuscript. This article is part of the last author’s Master’s thesis. She gratefully acknowledges the Department of Mathematics at the Universidade Estadual de Campinas (UNICAMP), Brazil, for providing a stimulating and productive research environment during her Master’s studies.


\begin{thebibliography}{99}


\bibitem{Bao2002} Bao, J.
\textit{Fully nonlinear elliptic equations on general domains}.
Canad. J. Math. 54 (2002), no. 6, 1121–1141.

\bibitem{BhatWarr2021} Bhattacharya, A. and Warren, M.
\textit{$C^{2,\alpha}$ estimates for solutions to almost linear elliptic equations}.
Commun. Pure Appl. Anal. 20 (2021), no. 4, 1363–1383.

\bibitem{BNV1994} Berestycki, H., Nirenberg, L., and Varadhan, S. R.S.
\textit{The principal eigenvalue and maximum principle for second-order elliptic operators in general domains}.
Comm. Pure Appl. Math. 47 (1994), no. 1, 47–92.


\bibitem{CabreCaff2003} Cabr\'{e}, X. and Caffarelli, L.A.
\textit{Interior $C^{2,\alpha}$ regularity theory for a class of nonconvex fully nonlinear elliptic equations}.
J. Math. Pures Appl. (9) 82 (2003), no. 5, 573–612.

\bibitem{Caffarelli1989} Caffarelli, L.A.
\textit{Interior a priori estimates for solutions of fully nonlinear equations}.
Ann. of Math. (2) 130 (1989), no. 1, 189–213.

\bibitem{CafCabre1995} Caffarelli, L.A. and Cabr\'{e}, X.
\textit{Fully nonlinear elliptic equations}.
Amer. Math. Soc. Colloq. Publ., 43
American Mathematical Society, Providence, RI, 1995. vi+104 pp. ISBN:0-8218-0437-5

\bibitem{CCKS} Caffarelli, L.A., Crandall, M.G.,  Kocan, M. and \'{S}wi\c{e}ch, A.
\textit{On viscosity solutions of fully nonlinear equations with measurable ingredients}.
Comm, Pure Appl. Math. 49 (1996) (4), 365--397.

\bibitem{Campanato1963} Campanato, S.
\textit{Propriet\`{a} di h\"{o}lderianit\`{a} di alcune classi di funzioni}.
Ann. Scuola Norm. Sup. Pisa Cl. Sci. (3) 17 (1963), pp. 175–188.


\bibitem{Campanato1964} Campanato, S.
\textit{Propriet\`{a} di una famiglia di spazi funzionali}.
Ann. Scuola Norm. Sup. Pisa Cl. Sci. (3) 18 (1964), 137–160.

\bibitem{CLW2011} Cao, Y., Li, D. and Wang, L.
\textit{A priori estimates for classical solutions of fully nonlinear elliptic equations}.
Sci. China Math. 54 (2011), no. 3, 457–462.

\bibitem{DaSilva-PhDThesis} Da Silva, J.V. 
\textit{Sharp and improved regularity estimates to fully nonlinear equations
and free boundary problems}. Tese de doutorado. Fortaleza: Universidade Federal do Cear\'{a},
Centro de Ci\^{e}ncias, Programa de P\'{o}s-Gradua\c{c}\~{a}o em Matem\'{a}tica, p. 105  (2019).

\bibitem{daSdosP2019} Da Silva, J.V. and Dos Prazeres, D.
\textit{Schauder type estimates for ``flat'' viscosity solutions to non-convex fully nonlinear parabolic equations and applications}.
Potential Anal. 50 (2019), no. 2, 149–170.

\bibitem{JVGN21}  Da~Silva, J.~V. and Nornberg, G. \textit{Regularity estimates for fully nonlinear elliptic PDEs with general Hamiltonian terms and unbounded ingredients}. Calc. Var. Partial Differential Equations {\bf 60} (2021), no.~6, Paper No. 202, 40 pp. 

 \bibitem{João} Da Silva, J.V. and Ricarte, G.C. \textit{Regularidade El\'{i}ptica e Problemas de Fronteiras Livres}. (Portuguese)[Elliptic regularity and free boundary problems]
34º Colóq. Bras. Mat., 8
Instituto Nacional de Matemática Pura e Aplicada (IMPA), Rio de Janeiro, 2023. 156 pp.
ISBN:978-85-244-0532-7.


\bibitem{dosPrazTei2016} Dos Prazeres, D. and Teixeira, E.V.
\textit{Asymptotics and regularity of flat solutions to fully nonlinear elliptic problems}.
Ann. Sc. Norm. Super. Pisa Cl. Sci. (5) 15 (2016), 485–500.

\bibitem{Evans1982} Evans, L.C.
\textit{Classical solutions of fully nonlinear, convex, second-order elliptic equations}.
Comm. Pure Appl. Math. 35 (1982), no. 3, 333–363.


 \bibitem{FR-O} Fern\'{a}ndez-Real, X.; Ros-Oton, X. 
        \textit{Regularity Theory for Elliptic PDE}. Zur. Lect. Adv. Math., 28
EMS Press, Berlin, [2022], ©2022. viii+228 pp.
ISBN:978-3-98547-028-0.

\bibitem{Fefferman2009} Fefferman, C. \textit{Extension of \(C^{m,\omega}\)- smooth functions by linear operators}. Rev. Mat. Iberoamericana. 25 (2009), 1–48.

\bibitem{Friedman1964} Friedman,  A. 
\textit{Partial differential equations of parabolic type}. Prentice-Hall,
Inc., Englewood Cliffs, NJ, 1964, pp. xiv+347.

\bibitem{Gilbarg-Tru2001} Gilbarg d. and Trudinger, N.~S. \textit{Elliptic partial differential equations of second order}. Reprint of the 1998 edition, Classics in Mathematics, Springer, Berlin, 2001.

\bibitem{Goffi2024} Goffi, A.
\textit{High-order estimates for fully nonlinear equations under weak concavity assumptions}.
J. Math. Pures Appl. (9) 182 (2024), 223–252.

\bibitem{Han2000}  Han, Q. 
\textit{Schauder estimates for elliptic operators with applications to nodal
sets}.
J. Geom. Anal. 10.3 (2000), pp. 455–480.

\bibitem{Han-Lin-Book} Han, Q. and Lin, F. \textit{Elliptic partial differential equations}. Vol. 1. Courant Lecture Notes in Mathematics. New York University, Courant Institute of Mathematical Sciences, New York; American Mathematical Society, Providence, RI,
1997, pp. x+144.

\bibitem{HuangAIHP} Huang, Q. \textit{Regularity theory for $L^n$-viscosity solutions to fully nonlinear elliptic equations with asymptotical approximate convexity}. Ann. Inst. H. Poincar\'{e} C Anal. Non Lin\'eaire  36 (2019), no.~7, 1869--1902.

\bibitem{HuangPAMS} Huang, Q.
\textit{On the regularity of solutions to fully nonlinear elliptic equations via the Liouville property}.  Proc. Amer. Math. Soc. 130 (2002), no. 7, 1955–1959.

\bibitem{Kichenassamy2006} Kichenassamy, S. 
\textit{Chapter 5 Schauder type estimates and applications}. Handbook of Differential Equations. Vol. 3, pp. 401–464, Elsevier (2006) ISBN: 978-044452846-9.

\bibitem{Kovats97} Kovats, J.
\textit{Fully nonlinear elliptic equations and the Dini condition}. Comm. PDE, 22, 1911–1927 (1997).

\bibitem{Kovats99} Kovats, J.
\textit{Dini-Campanato spaces and applications to nonlinear elliptic equations}. Electronic JDE, 37, 1–20 (1999). 

\bibitem{Krylov1983} Krylov, N. V.
\textit{Boundedly inhomogeneous elliptic and parabolic equations in a domain}.
Izv. Akad. Nauk SSSR Ser. Mat. 47 (1983), no. 1, 75–108.

\bibitem{Kry96} Krylov, N.~V. {\it Lectures on elliptic and parabolic equations in H\"older spaces}, Graduate Studies in Mathematics, 12, Amer. Math. Soc., Providence, RI, 1996.

\bibitem{Lian2024} Lian, Y. \textit{Interior point-wise regularity for elliptic and parabolic
equations in divergence form and applications to nodal sets}. Preprint, arXiv:2405.07214 (2024).

\bibitem{Lian-Wang-Zhang2024} Lian, Y., Wang, L., and Zhang, K. \textit{Point-wise Regularity for Fully Nonlinear Elliptic Equations in General Forms}. Preprint, arXiv:2012.00324v2 (2024).

\bibitem{Lian-Zhang2023} Lian, Y. and Zhang, K. 
\textit{Boundary point-wise regularity and applications to the regularity of free boundaries}. 
Calc. Var. Partial Differential Equations 62.8
(2023), Paper No. 230, 32. 

\bibitem{Lind-Monn2013} Lindgren, E. and  Monneau, R. 
\textit{Pointwise estimates for the heat equation. Application to the free boundary of the obstacle problem with Dini coefficients.} Indiana Univ. Math. J. 62.1 (2013), pp. 171–199.

\bibitem{Lind-Monn2015} Lindgren, E. and  Monneau, R. 
\textit{Pointwise regularity of the free boundary for the parabolic obstacle problem}. Calc. Var. Partial Differential Equations 54.1
(2015), pp. 299–347.


\bibitem{Monneau2009} Monneau, R.
\textit{Pointwise estimates for Laplace equation. Applications to the free boundary of the obstacle problem with Dini coefficients.}
J. Fourier Anal. Appl. 15 (2009), no. 3, 279–335.

\bibitem{NV2007} Nadirashvili, N. and Vl\u adu\c t, S.
\textit{Nonclassical solutions of fully nonlinear elliptic equations}.
Geom. Funct. Anal. 17 (2007), no. 4, 1283-1296.

\bibitem{NV2008} Nadirashvili, N. and Vl\u adu\c t, S.
\textit{Singular viscosity solutions to fully nonlinear elliptic equations}.J. Math. Pures Appl. (9) 89 (2008), no. 2, 107-113.


\bibitem{Nirenberg1953} Nirenberg, L.
\textit{On nonlinear elliptic partial differential equations and H\"{o}lder continuity}.
Comm. Pure Appl. Math. 6 (1953), 103–156.

\bibitem{Peetre1966} Peetre, J.
\textit{On convolution operators leaving $L^{p,\lambda}$ spaces invariant}.
Ann. Mat. Pura Appl. (4) 72 (1966), 295–304.

\bibitem{PetSHaUral12} Petrosyan, A., Shahgholian, H. and Uraltseva, N.~N. {\it Regularity of free boundaries in obstacle-type problems}, Graduate Studies in Mathematics, 136, Amer. Math. Soc., Providence, RI, 2012.

\bibitem{Safonov1984} Safonov, M. V.
\textit{The classical solution of the elliptic Bellman equation}, 
Dokl. Akad. Nauk SSSR, 278, 1984, 810–813; English translation in Soviet Math Doklady, 30, 1984, 482–485.

\bibitem{Safonov1989} Safonov, M. V.
\textit{Classical solution of second-order nonlinear elliptic equations}.
Izv. Akad. Nauk SSSR Ser. Mat. 52 (1988), no. 6, 1272–1287, 1328; translation in Math. USSR-Izv. 33 (1989), no. 3, 597–612.

\bibitem{Savin2007} Savin, O.
\textit{Small perturbation solutions for elliptic equations}.
Comm. Partial Differential Equations 32 (2007), no. 4-6, 557–578.

\bibitem{Sim97} Simon, L. 
\textit{Schauder estimates by scaling}, Calc. Var. Partial Differential Equations
5 (1997), 391-407.

\bibitem{Teixeira2016} Teixeira, E.V.
\textit{Geometric regularity estimates for elliptic equations}. Mathematical Congress of the Americas, 185–201. Contemp. Math., 656
American Mathematical Society, Providence, RI, 2016
ISBN:978-1-4704-2310-0.

\bibitem{Teixeira2020} Teixeira, E.V.
\textit{Regularity theory for nonlinear diffusion processes}.
Notices Amer. Math. Soc. 67 (2020), no. 4, 475–483.


\bibitem{Trud1986} Trudinger, N.S.
\textit{A new approach to the Schauder estimates for linear elliptic equations}. Miniconference on operator theory and partial differential equations (North Ryde, 1986), 52–59.
Proc. Centre Math. Anal. Austral. Nat. Univ., 14
Australian National University, Centre for Mathematical Analysis, Canberra, 1986. ISBN:0-86784-517-1.

\bibitem{Wang2006} Wang, X.-J. \textit{Schauder estimates for elliptic and parabolic equations}, 
Chinese Ann. Math. Ser. B 27 (2006), 637-642.

\bibitem{WuNiu2023} Wu, D. and Niu, P.
\textit{Interior point-wise $C^{2,\alpha}$ regularity for fully nonlinear elliptic equations}.
Nonlinear Anal. 227 (2023), Paper No. 113159, 9 pp.

\bibitem{WH2022} Wu, Y. and Yu, H.
\textit{On the fully nonlinear Alt-Phillips equation}.
Int. Math. Res. Not. IMRN 2022, no. 11, 8540–8570.

\bibitem{ZC2002} Zou, X. and Chen, Y.Z.
\textit{Fully nonlinear parabolic equations and the Dini condition}.
Acta Math. Sin. (Engl. Ser.) 18 (2002), no. 3, 473–480.

\end{thebibliography}
\end{document}